\documentclass[12pt]{amsart}

\RequirePackage[OT1]{fontenc}
\RequirePackage{amsthm,amsmath,amssymb}
\RequirePackage[numbers]{natbib}
\RequirePackage[colorlinks,citecolor=blue,urlcolor=blue]{hyperref}
\usepackage{xcolor}
\usepackage{epsfig}
\usepackage{graphicx}
\usepackage{epstopdf}
\usepackage{multirow}
\usepackage{listings}
\numberwithin{equation}{section}
\theoremstyle{plain}
\newtheorem{lemma}{Lemma}[section]
\newtheorem{theorem}{Theorem}[section]

\newtheorem{corollary}{Corollary}[section]

\newcommand{\bu}{{\mathbf u}}

\newcommand{\T}{{\mathbf T}}
\newcommand{\A}{{\mathbf A}}
\newcommand{\B}{{\mathbf B}}
\newcommand{\C}{{\mathbf C}}
\newcommand{\D}{{\mathbf D}}
\newcommand{\bS}{{\mathbf S}}
\newcommand{\bx}{{\mathbf x}}
\newcommand{\y}{{\mathbf y}}
\newcommand{\z}{{\mathbf z}}
\newcommand{\br}{{\mathbf r}}

\newcommand\tx{\widetilde{\bx}}

\newcommand{\CC}{\mathbb{C}}

\newcommand{\um}{\underline{m}}

\newcommand{\uS}{\underline{S}}
\newcommand{\uT}{\underline{T}}
\newcommand{\de}{\delta}
\newcommand{\E}{{\rm E}}
\newcommand{\Cov}{{\rm Cov}}
\newcommand{\bL}{{\mathbf L}}
\newcommand{\bV}{{\mathbf V}}
\newcommand{\bM}{{\mathbf M}}
\newcommand{\bR}{{\mathbf R}}

\newcommand{\R}{{\tilde {\mathbf R}}}
\newcommand{\tr}{{\text{\rm tr}}}
\newcommand{\btheta}{{\boldsymbol \theta}}
\newcommand{\bbeta}{{\boldsymbol \beta}}
\newcommand{\bgamma}{{\boldsymbol \gamma}}

\begin{document}
\title[High-dimensional spatial-sign covariance matrix]{On spectral properties of high-dimensional spatial-sign covariance matrices in elliptical distributions with  applications}

\author{Weiming Li  \quad Wang Zhou}
\date{}
\thanks{Li's work was partially supported by National Natural Science Foundation of China, No.\ 11401037 and Program of IRTSHUFE.Zhou's work was partially supported by the MOE Tier 2 grant MOE2015-T2-2-039 (R-155-000-171-112)  at the National University of Singapore.}

\address{School of Statistics and Management, Shanghai University of Finance and Economics, Guoding Road No.\ 777, Shanghai, 200433, China.}
\email{li.weiming@shufe.edu.cn}

\address{Department of Statistics and Applied Probability, National University of Singapore, Singapore}
\email{stazw@nus.edu.sg}



\maketitle

\begin{abstract}
Spatial-sign covariance matrix (SSCM) is an important substitute  of sample covariance matrix (SCM) in robust statistics.
This paper investigates the SSCM on its asymptotic spectral behaviors under high-dimensional elliptical populations, where both the dimension $p$ of observations and the sample size $n$ tend to infinity with their ratio $p/n\to c\in (0, \infty)$. The empirical spectral distribution of this nonparametric scatter matrix is shown to converge in distribution to a generalized Mar\v{c}enko-Pastur law. Beyond this, a new central limit theorem (CLT) for general linear spectral statistics of the SSCM is also established. For polynomial spectral statistics, explicit formulae of the limiting mean and covarance functions in the CLT are provided. The derived results are then applied to an estimation procedure and a test procedure for the spectrum of the shape component of population covariance matrices.
\end{abstract}

\section{Introduction}
Elliptical family of distributions, originally introduced in \cite{K70}, is an important extension of the multivariate normal distribution and has been broadly applied in biology, finance and economics, signal and image processing, etc. \citep{FZ90,G13}. A random vector $\bx$ with zero mean is said to be elliptically distributed if it has a stochastic representation \citep{FZ90}:
\begin{equation}\label{eds}
\bx =w \A\bu,
\end{equation}
where 
$\A$ is a $p\times p$ matrix with $rank(\A)=p$, $w\geq 0$ is a scalar random variable representing the radius of $\bx$,  and $\bu \in \mathbb R^{p}$ is the random direction, independent of $w$ and uniformly distributed on the unit sphere in $\mathbb R^p$.
Besides the normal distribution, this family includes many other celebrated distributions, such as multivariate $t$-distribution, Kotz-type distributions, and Gaussian scale mixture. In general, the radius $w$ needs not be independent of the direction $\bu$ but can be a function of the chosen direction \citep{R89}.

Let $\bx_1,\ldots, \bx_n$ be a sequence of independent and identically distributed (i.i.d.) random vectors from the elliptical model in \eqref{eds}. Many statistical procedures for this model prefer to transform the original observations into spatial-sign samples for the purpose of robustness, which are defined as
$$
\y_j=
\begin{cases}
\sqrt{p}\frac{\bx_j}{||\bx_j||}& \bx_j\neq 0,\\
0&\bx_j=0.
\end{cases}
$$
One can refer to \cite{M06} and \cite{O10} for a comprehensive review.
When an inference is concerned with the shape matrix $\T=\A\A'$, assuming $\tr(\T)=p$ so that $w$ and $\A$ can be identified in the model \eqref{eds}, one of the most important statistics is the so-called {\em spatial-sign covariance matrix} (SSCM), i.e.
$$
\B_{n}=\frac{1}{n}\sum_{j=1}^n\y_j\y_j',
$$
which is actually the sample covariance matrix (SCM) of $(\y_j)$.
As a robust alternative to the SCM $\bS_n=\sum_{j=1}^n\bx_j\bx_j'/n$, this nonparametric scatter matrix $\B_n$ is a fast computed and orthogonally equivariant statistic with high breakdown point, and thus is highly recommended in applications, such as principle component analysis and structural test for covariance matrices, see \cite{L99}, \cite{G08}, \cite{V11}, \cite{PV16}, to name a few. Despite its merits, the SSCM is also a controversial statistic in `` small $p$, large $n$" scenarios due to its lack of affine equivariance \citep{M14}.
However, the pursuit of this property seems not advisable for high-dimensional situations, as claimed in \cite{T10} that any well-defined affine equivariant scatter matrix must be proportional to the SCM $\bS_n$ whenever $p>n$. Therefore,
it is of great interests to discover behaviors of the SSCM in high-dimensional robust statistics.

In this paper, using tools of random matrix theory, we investigate asymptotic spectral behaviors of the SSCM $\B_n$ in high-dimensional frameworks where both the dimension $p$ and the sample size $n$ tend to infinity with their ratio $p/n\rightarrow c$, a positive constant in $(0, \infty)$.
Specifically,
let $(\lambda_j)_{1\le j\le p}$ be the eigenvalues of $\B_n$, 
then the
{\em empirical spectral distribution} (ESD) of $\B_n$ is by
definition
\[
F^{\B_n}=\frac{1}{p}\sum_{j=1}^p\delta_{\lambda_j},
\]
where  $\delta_b$ denotes  the Dirac mass at $b$.  Our aim is to study the limiting properties of $F_n$ and the central limit theorem (CLT) for {\em linear spectral statistics} (LSS) of the form $\int f(x) dF_n(x)$ for a class of smooth test functions $f$.
These properties may become powerful tools to recover spectral features of the population SSCM, i.e. $\Sigma=p\E(\bx\bx'/||\bx||^2)$, and then those of the
shape matrix $\T$ since the matrices $\Sigma$ and $\T$ share the same eigenvectors and their eigenvalues have a one-to-one correspondence \citep{D16}. Moreover, as $p\rightarrow\infty$, the two matrices coincide in the sense that the spectral norm $||\Sigma-\T||\to0$, as long as $||\Sigma||$ (or $||\T||$) is uniformly bounded, see Lemma \ref{mom-y}.

Spectral properties of high-dimensional SCM have been extensively studied in random matrix theory since
the pioneer work of \cite{MP67}.
The standard model in the literature has the form
\begin{equation}\label{linear-trans}
\tx=\sigma\A \z,
\end{equation}
where $\A$ is as before, $\sigma$ is a constant, and $\z=(z_1,\ldots,z_p)'\in\mathbb R^p$ is a set of i.i.d.\ random variables satisfying E$(z_1)=0$, E$(z_1^2)=1$, and E$(z_1^4)<\infty$.
Let $\tx_1,\ldots,\tx_n$ be $n$ i.i.d.\ copies of $\tx$ and
$\widetilde \bS_n=\sum_{j=1}^n \tx_j\tx_j'/n$ be the corresponding SCM. It has been known that the ESD of $\widetilde \bS_n$ converges to the celebrated
Mar\v{c}enko-Pastur (MP)  law when $\A=I_p$, and generalized MP law for general matrix $\A$, as $(n,p)\to\infty$ with $p/n\to c>0$. One can refer to \cite{MP67} and \cite{S95}. The CLT for LSS of $\widetilde\bS_n$ was first studied in \cite{Jonsson82} by assuming the population to be standard multivariate normal.
One breakthrough on the CLT was obtained by \cite{BS04}, where the population is allowed to be general with E$(z_1^4)=3$.
This fourth moment condition was then weakened to be E$(z_1^4)<\infty$ in \cite{PZ08}. For more references, one can refer to \cite{BSbook}, \cite{BJPZ15}, \cite{Gao16}, and references therein.
However, these results do not apply to general elliptical populations since the two underlying models in \eqref{eds} and \eqref{linear-trans} have little in common, except for normal distributions.
In fact, for general elliptical populations, it has been reported that the ESD of the SCM $\bS_n$ converges to a deterministic distribution that is not a generalized MP law, but has to be characterized by both the distribution of $w$ and the limiting spectrum of $\T$ through a system of implicit equations \citep{E09,LY16}. The involvement of $w$ seriously interferes with our understanding of the spectrum of $\T$ from the ESD of $\bS_n$. This again motivates us to shift our attention to the SSCM $\B_n$ which discards the random radiuses $(w_j)$ and focus only on the directions $(\A\bu_j)$.

The main contributions of this paper are as follows. First in
Section~\ref{sec:main}, asymptotic results on the eigenvalues of $\B_n$ are derived, including the limit of the ESD $F_n$ and a new CLT for LSS of $\B_n$. As a corollary, polynomial spectral statistics are fully addressed with explicit limiting mean and covariance functions in the CLT.
Then in  Section~\ref{sec:app}, relying on these results,
we develop two statistical applications on the spectrum of $\Sigma$, the population SSCM, under a setting that the spectrum forms a  discrete distribution with finite support. One is to estimate the spectrum of $\Sigma$ through moment methods and the other is to test the hypothesis that there are no more than $d_0$ distinct eigenvalues of $\Sigma$. Technical proofs of the main theorems are gathered in Section~\ref{sec:proofs}. Some lemmas and their necessary proofs are postponed to the last section.

\section{High-dimensional theory for eigenvalues of $\B_{n}$}
\label{sec:main}

\subsection{Limiting spectral distribution of $\B_{n}$}

We consider here the limit of the ESD sequence $(F^{\B_n})$ in high-dimensional regimes, namely {\em limiting spectral distribution} (LSD).
Our main assumptions are listed below.

\medskip
\noindent{\em Assumption}   (a). \quad
Both the sample size and population dimension $n,p$ tend to infinity in such a way that $c_n=p/n\to c \in(0,\infty)$.

\medskip
\noindent{\em Assumption}   (b). \quad
Sample observations are $\y_j= \sqrt{p}\A\bu_j/||\A\bu_j||$, $j=1,\ldots,n,$
where $\A$ is a $p\times p$ matrix with $\A\A'=\T$ and $(\bu_j)$ are i.i.d.\ random vectors, uniformly distributed on the unit sphere in $\mathbb R^p$.

\medskip
\noindent{\em Assumption}   (c). \quad
The spectral norm of $\Sigma=\E(\y_1\y_1')$ is bounded and its spectral distribution $H_p$ converges weakly to a probability distribution $H$, called {\em population spectral distribution} (PSD).
\medskip

From Lemma \ref{mom-y}, it is clear that the spectral distributions of $\Sigma$ and $\T$  are asymptotically identical. So one can certainly replace $\Sigma$ with $\T$ in Assumption (c), which does not affect the LSD of $F^{\B_n}$.
However we keep $\Sigma$ because it is easy to describe the CLT for LSS using the spectral distribution $H_p$ of $\Sigma$.

For the characterization of the LSD of $F^{\B_n}$, we need to introduce the Stieltjes transform of a measure $G$ on the real line, which is defined as
$$m_G(z)=\int\frac{1}{x-z}dG(x),\quad z\in\mathbb C\setminus S_G,$$
where $S_G\subset \mathbb R$ denotes the support of $G$.
\begin{theorem}\label{lsd}
	Suppose that Assumptions (a)-(c) hold. Then, almost surely, the empirical spectral distribution $F^{\B_n}$ converges weakly to a probability distribution $F^{c, H}$,
	whose Stieltjes transform $m=m(z)$ is the unique  solution to the equation
	\begin{eqnarray}\label{mp1}
	m=\int\frac{1}{t(1-c-czm)-z}d H(t)~,\quad z\in\CC^+,
	\end{eqnarray}
	in the set $\{m\in \mathbb C:  -(1-c)/z+cm\in{\mathbb C^+}\}$ where $\mathbb C^+\equiv\{z\in \mathbb C: \Im(z)>0\}$.
\end{theorem}

The LSD $F^{c, H}$ defined in \eqref{mp1} agrees with that in \cite{MP67}.
Let $\um=\um(z)$ denote  the Stieltjes transform of $\underline{F}^{c, H}=cF^{c,H} + (1-c)\de_0$.  Then \eqref{mp1} can also be represented as
\begin{equation}  \label{mp}
z  =  - \frac1 {\um}  +  c \int\!\frac{t}{1+t\um} d H(t)~,\quad z\in\CC^+.
\end{equation}
See \cite{S95}. For procedures on finding the density function and the support set of $F^{c, H}$ from \eqref{mp1} and \eqref{mp}, one is referred to  \cite{BSbook}.

\subsection{CLT for linear spectral statistics of $\B_{n}$}

Let $F^{c_n, H_p}$ be the LSD as defined in \eqref{mp} with the parameters $(c,H)$ replaced by $(c_n,  H_p)$.
Writing $G_n=F^{\B_n}-F^{c_n,H_p}$, we next study the fluctuation of
\begin{eqnarray*}
	\int f(x) dG_{n}(x)=\int f(x) d[F^{\B_{n}}(x)-F^{c_n,  H_p}(x)],
\end{eqnarray*}
which is a centralized linear spectral statistic with analytic $f$.

\begin{theorem}\label{clt}
	Suppose that Assumptions (a)-(c) hold.  Let $f_1,\ldots, f_k$ be $k$ functions analytic on an open interval containing
	\begin{equation*}\label{interval}
	\left[\liminf_{p\rightarrow\infty}\lambda_{\min}^{\Sigma}\delta_{(0,1)}(c)(1-\sqrt{c})^2,\quad \limsup_{p\rightarrow\infty}\lambda_{\max}^{\Sigma}(1+\sqrt{c})^2\right].
	\end{equation*}
	Then the random vector
	$$
	p\left(\int f_1(x)dG_n(x),\ldots, \int f_k(x)dG_n(x)\right)
	$$
	converges weakly to a Gaussian vector $(X_{f_1},\ldots,X_{f_k})$, whose mean function is
	\begin{align*}
	{\E}X_f
	=&-\frac{1}{2\pi\rm i}\oint_{\mathcal C_1} f(z)\int\frac{c(\um'(z)t)^2d H(t)}{\um(z)(1+\um(z) t)^3}dz
	-\frac{c\um(z)\um'(z)}{\pi\rm i}\oint_{\mathcal C_1} f(z)\\
	&\times\bigg[\int\frac{ \gamma_2t-t^2d H(t)}{1+\um(z) t}\int\frac{td H(t)}{(1+\um(z) t)^2}-\int\frac{td H(t)}{1+\um(z) t}\int\frac{t^2d H(t)}{(1+\um(z) t)^2}\bigg]dz
	\end{align*}
	and covariance function is
	\begin{align*}
	{\rm Cov}\left(X_f,X_g\right)
	=&-\frac{1}{2\pi^2}\oint_{\mathcal C_1}\oint_{\mathcal C_2}\frac{f(z_1)g(z_2)\um'(z_1)\um'(z_2)}{(\um(z_1)-\um(z_2))^2}dz_1dz_2\nonumber\\
	&+2 \gamma_2c\int xf'(x)dF(x)\int xg'(x)dF^{c,H}(x)\\
	&-\frac{1}{\pi{\rm i}}\oint_{\mathcal C_1} \frac{f(z)\um'(z)}{\um^2(z)}dz\int xg'(x)dF^{c,H}(x)\\
	&-\frac{1}{\pi{\rm i}}\oint_{\mathcal C_1} \frac{g(z)\um'(z)}{\um^2(z)}dz\int xf'(x)dF^{c,H}(x),
	\end{align*}
	$(f, g \in \{f_1,\cdots,f_k\})$, where the contours $\mathcal C_1$ and $\mathcal C_2$ are non-overlapping, closed, counter-clockwise orientated in the complex plane, and each encloses the support of
	the LSD $F^{c,H}$.
\end{theorem}

When the underlying population is multivariate normal,  the elliptical model in \eqref{eds} and the linear transformation model in \eqref{linear-trans} hold simultaneously. In this case, it is interesting to compare the limiting distribution in Theorem \ref{clt} based on SSCM with the classical result in \cite{BS04} based on SCM. It turns out that there are some additional terms in our new CLT: the second contour integral in the mean function and the second to fourth summands in the covariance function.

Among all LSS, polynomial spectral statistics are of fundamental importance.
The bases of these statistics are
moments of ESD $F^{\B_n}$, i.e.
\begin{equation*}\label{beta:nk}
\hat \beta_{nj}=\frac{1}{p}\tr(\B_n^j)=\int x^jdF^{\B_n}(x), \quad j=1,2,\ldots.
\end{equation*}
The first order moment $\hat \beta_{n1}$ is 1 since $\tr(\B_n)\equiv\tr(\Sigma)\equiv p$.
Other moments $(\hat \beta_{nj})$, $j\geq 2$, are random. Their limiting behavior can be described through the following two  quantities
\begin{eqnarray*}\label{betan}
	\beta_{nj}=\int x^jdF^{c_n,  H_p}(x)\quad\text{and}\quad  \gamma_{nj}=\int t^jd  H_p(t),
\end{eqnarray*}
as well as their limits, denoted by $\beta_j$ and $\gamma_j$, respectively, $j=1,2,\ldots.$
From  \cite{NS06}, the quantities $(\beta_{nj})$ and
$(\gamma_{nj})$ are connected through the recursive formulae:
\begin{eqnarray}\label{beta-gamma}
\beta_{nj}=
\sum c_n^{i_1+\cdots+i_j-1}( \gamma_{n2})^{i_2}\cdots( \gamma_{nj})^{i_{j}}\phi(i_1,\ldots,i_{j}),\quad j\geq2,
\end{eqnarray}
and $\beta_{n1}=\gamma_{n1}\equiv1$, where the sum runs over the following partitions of $j$:
$$
(i_1,\ldots,i_{j}): j=i_1+2i_2+\cdots+ji_{j}, \quad i_l\in\mathbb N,
$$
and $\phi(i_1,\ldots,i_{j})=j!/[i_1!\cdots i_{j}!(j+1-i_1-\cdots -i_{j})!].$
The joint limiting distribution of moments $(\hat \beta_{nj})_{2\leq j\leq k}$ can be derived from Theorem \ref{clt} by taking functions $f_j(x)=x^j,
j=2,\ldots,k$. For this particular case, the mean and covariance functions in the limiting distribution can be explicitly formulated.

\begin{corollary}\label{clt-mom}
	Suppose that Assumptions (a)-(c) hold. Then the random vector
	\begin{equation*}
	p\left(\hat\beta_{n2}-\beta_{n2},\ldots, \hat\beta_{nk}-\beta_{nk}\right)\xrightarrow{D}N_{k-1}(v,\Psi).
	\end{equation*}
	The mean vector $v=(v_j)_{2\leq j\leq k}$ satisfies
	\begin{eqnarray*}
		v_j=
		\bigg[\frac{cP^j}{(j-2)!}\bigg(\frac{P_{2,3}}{1-cz^2P_{2,2}}+2 \gamma_2P_{1,1}P_{1,2}-2P_{2,1}P_{1,2}-2P_{1,1}P_{2,2}\bigg)\bigg]^{(j-2)}\bigg|_{z=0},
	\end{eqnarray*}
	where $P_{s,t}=\int x^s(1+xz)^{-t}d H(x)$, $P=(czP_{1,1}-1)$,
	and $g^{(\ell)}(z)$ denotes the $\ell$th derivative of $g(z)$ with respect to $z$.
	The covariance matrix $\Psi=(\psi_{ij})_{2\leq i, j\leq k}$ has  entries
	\begin{eqnarray*}
		\psi_{ij}=2\sum_{\ell=0}^{i-1}(i-\ell)u_{i,\ell}u_{j,i+j-\ell}+2c \gamma_2ij\beta_i\beta_j+2j\beta_ju_{i,i+1}+2i\beta_iu_{j,j+1},
	\end{eqnarray*}
	where $u_{s,t}=[P^s]^{(t)}/t!|_{z=0}$.
	
\end{corollary}

\section{Applications to spectral inference}
\label{sec:app}

Inference on PSD is fundamentally important in many high-dimensional statistical analysis, such as the principal component analysis \citep{J01,C13,WF17}, factor models \citep{F08,F13}, and covariance matrix estimation \citep{LW12}.

In this section, we illustrate two statistical applications of the theoretical results developed in Section \ref{sec:main}: one is estimating a PSD and the other is testing the order of a PSD.
The family of PSDs under study is a class of parameterized discrete distributions with finite support on $\mathbb R^+$, that is,
\begin{equation}\label{psd}
H(\btheta)=w_1\delta_{a_1}+\cdots+ w_d\delta_{a_d},\quad \btheta=(a_1, w_1,\ldots, a_{d-1}, w_{d-1})\in{\Theta},
\end{equation}
where ${\Theta} =\left\{\btheta: 0< a_1<\cdots <a_d<\infty;\ 0< \prod_{i=1}^d w_i, \ \sum_{i=1}^{d}a_i^\ell w_i=1, \ell = 0, 1\right\}.$
Here the restriction $\sum_{i=1}^da_iw_i=1$ is due to the fact that $\int tdH_p(t)=\tr(\Sigma)/p\equiv1$. For the model \eqref{psd}, the order of $H$ refers to the cardinality of its support, which is equal to $d$.
This model for PSDs can be viewed as the spectral structure of noise covariance matrices in factor models \citep{F08}, and
extensions of the spiked model \citep{J01} which allows the number of leading eigenvalues to grow with the dimension $p$.
More discussions on this model can be found in  \cite{E08}, \cite{R08}, \cite{B10}, \cite{L14}, etc. Similar to \cite{E08}, we adopt the setting of fixed PSDs in this section, i.e.  $(c_n, H_p)\equiv (c, H)$ for all $(n,p)$ large.

\subsection{Estimation of a PSD}
For the model in \eqref{psd}, \cite{B10} introduced a moment method for the PSD estimation.
By assuming the order $d$ to be known, their method first estimates the moments $(\gamma_j)$ of $H$ through the recursive formulae in \eqref{beta-gamma}, and then solve a system of moment equations,
$
\{\hat\gamma_j=\sum_{i=1}^da_i^jw_i,\ j=0,\ldots,2d-1\},
$
to get a consistent estimator of $\btheta$.

In our situation, with notation $\bbeta_j=(\beta_2,\ldots,\beta_j)'$ and $\bgamma_j=(\gamma_2,\ldots,\gamma_j)'$ for $j\geq 2$, we denote
\begin{equation*}\label{maps}
g_1:\bgamma_{2d-1}\rightarrow\btheta \quad\text{and}\quad g_{2,j}:\bbeta_j\rightarrow \bgamma_j
\end{equation*}
as the mappings between the corresponding vectors. These two mappings are both one-to-one and the determinants of their Jacobian matrices are all nonzero. See \cite{B10}.  Therefore, applying Theorem \ref{lsd},
$\hat\bbeta_j:=(\hat\beta_{n2},\ldots,\hat\beta_{nj})'\xrightarrow{a.s.}\bbeta_j$ which is followed by $\hat\btheta_n:=g_1\circ g_{2,2d-1}(\hat\bbeta_{2d-1})\xrightarrow{a.s.}\btheta$, as $(n,p)\to\infty$.
However, as shown by the CLT in Corollary \ref{clt-mom}, the estimator $\hat\bbeta_j$ is biased by the order of $O(1/p)$. So it's natural to modify $\hat\bbeta_j$ by subtracting its limiting mean in the CLT to obtain a better estimator of $\btheta$. Beyond this correction, the CLT can also provide confidence regions for the parameter $\btheta$.

Denote the modified estimators of $\bbeta_j$, $\bgamma_j$, and $\btheta$ by
\begin{equation}\label{cor-est}
\hat\bbeta_j^*=\hat\bbeta_j-\frac{1}{p}(\hat  v_2,\ldots,\hat v_j)',\quad \hat\bgamma_j^*= g_{2,j}(\hat\bbeta_j^*), \quad\text{and}\quad \hat\btheta_n^*=g_1(\hat\bgamma_{2d-1}^*),
\end{equation}
respectively, where $\hat v_\ell=v_\ell(\hat\bbeta_\ell)$ with $v_\ell$ defined in Corollary \ref{clt-mom} for $\ell=2,\ldots, j.$
From Theorem \ref{lsd}, Corollary \ref{clt-mom}, and a standard application of the Delta method, one may easily get asymptotic properties of these estimators.

\begin{theorem}\label{th-gamma}
	Suppose that  Assumptions  (a)-(c) hold and the true value $\btheta$ is an inner point of $\Theta$. Then we have
	$\hat\bbeta_j^*\xrightarrow{a.s.} \bbeta_j$,
	$\hat\bgamma_j^*\xrightarrow{a.s.} \bgamma_j$, $\hat\btheta_n^*\xrightarrow{a.s.}\btheta$, and
	moreover
	\begin{align}\label{clt-gamma}
	p\big(\hat\bgamma_j^*-\bgamma_{j}\big)&\xrightarrow{D} N_{j-1}(0,J_{2,j}\Psi_j J_{2,j}'),\\
	p\big(\hat\theta_n^*-\btheta\big)&\xrightarrow{D} N_{2k-2}(0,J_1J_{2,2d-1}\Psi_{2d-1}J_{2,2d-1}'J_1'),\nonumber
	\end{align}
	where $J_{1}$ and $J_{2,\ell}$ represent the Jacobian matrices $\partial g_{1}/\partial \bgamma_{2d-1}$  and $\partial g_{2,\ell}/\partial \bbeta_\ell$,
	respectively, and $\Psi_\ell$ is defined in Corollary \ref{clt-mom} with $k=\ell$.
\end{theorem}

\subsection{Test for the order of a PSD}

The aforementioned estimation procedure requires that the order $d$ of the PSD be pre-specified. In general, this prior knowledge should be testified in advance.
To deal with this problem, we consider the hypotheses
\begin{eqnarray}\label{hyp}
H_0: d\leq d_0\quad v.s.\quad H_1: d>d_0,
\end{eqnarray}
where $d_0\geq 1$ is a known constant. These hypotheses can also be regarded as a generalization of the well-known sphericity hypotheses on covariance matrices, i.e. the case $d_0=1$.

In \cite{Q16}, a test procedure was outlined based on a moment matrix $\Gamma$ and its estimator $\widehat\Gamma$ which can be formulated as
\begin{eqnarray*}\label{gamma-h}
	\Gamma= \left(\begin{matrix}
		1&\gamma_1&\cdots&\gamma_{d_0}\\
		\gamma_1&\gamma_2&\cdots&\gamma_{d_0+1}\\
		\vdots&\vdots&&\vdots\\
		\gamma_{d_0}&\gamma_{d_0+1}&\cdots&\gamma_{2d_0}
	\end{matrix}\right)\quad\text{and}\quad
	\widehat \Gamma=
	\left(\begin{matrix}
		1&\hat\gamma_1&\cdots&\hat\gamma_{d_0}\\
		\hat\gamma_1&\hat\gamma_2&\cdots&\hat\gamma_{d_0+1}\\
		\vdots&\vdots&&\vdots\\
		\hat\gamma_{d_0}&\hat\gamma_{d_0+1}&\cdots&\hat\gamma_{2d_0}
	\end{matrix}\right).
\end{eqnarray*}
Here we set $\hat{\gamma}_1=1$ and $\hat\gamma_j=\hat{\gamma}^*_j$, as defined in \eqref{cor-est}, for $j\geq 2$.
It has been proved that the determinant $\det(\Gamma)$ of $\Gamma$ is zero if the null hypothesis in \eqref{hyp} holds, otherwise $\det(\Gamma)$ is strictly positive \citep{L14}.
Therefore, the determinant $\det(\widehat\Gamma)$ can serve as a test statistic for \eqref{hyp} and the null hypothesis shall be rejected if the statistic is significantly greater than zero.
Applying Theorem \ref{th-gamma} and the main theorem in \cite{Q16}, the asymptotic distribution of $\det(\widehat\Gamma)$ is obtained immediately.

\begin{theorem}\label{th1}
	Suppose that Assumptions (a)-(c) hold. Then the statistic $\det(\widehat\Gamma)$ is asymptotically normal, i.e.
	\begin{equation}\label{clt-G}
	p\left(\det(\widehat\Gamma)-\det(\Gamma)\right)\xrightarrow{D} N(0,\sigma^2),
	\end{equation}
	where $\sigma^2=\alpha'V\Omega V'\alpha$ with $\alpha=vec(adj(\Gamma))$, the vectorization of the adjugate matrix of $\Gamma$.
	The first two rows and columns of the $(2d_0+1)\times (2d_0+1)$ matrix $\Omega$ consist of  zero and the remaining submatrix $J_{2,2d_0}\Psi_{2d_0} J_{2,2d_0}'$ is defined in \eqref{clt-gamma}.
	The $(d_0+1)^2\times (2d_0+1)$ matrix $V=(v_{ij})$ is a 0-1 matrix with only $v_{i,a_i}=1$, $a_i=i-\lfloor(i-1)/(d_0+1)\rfloor d_0$, $i=1,\ldots,(d_0+1)^2$, where $\lfloor x\rfloor$ denotes the greatest integer not exceeding $x$.
\end{theorem}

From Theorem \ref{th-gamma}, the limiting variance $\sigma^2$ in \eqref{clt-G} is a continuous function of $\bgamma_{4d_0}$. While, under the null hypothesis, this variance is a function of
$\bgamma_{2d_0-1}$, denoted by $\sigma^2_{H_0}(\bgamma_{2d_0-1})$.
Let $\hat\sigma_{H_0}^2=\sigma^2_{H_0}(\hat\bgamma^*_{2d_0-1})$. Then it is a strongly consistent estimator of $\sigma_{H_0}^2(\bgamma_{2d_0-1})$.

\begin{corollary}\label{th2}
	Suppose that Assumptions (a)-(c) hold. Then, under the null hypothesis,
	\begin{eqnarray*}
		T_n:=\frac{p\det(\widehat \Gamma)}{\hat\sigma_{H_0}}\xrightarrow{D} N(0,1),
	\end{eqnarray*}
	as $n\to\infty$. In addition, the asymptotic power of  $T_n$ tends to 1.
\end{corollary}

Corollary \ref{th2} follows directly from Theorem \ref{th1} and its proof is thus omitted. This corollary includes as a particular case the sphericity test. For this case,  the test statistic reduces to $T_n=n(\hat\gamma^*_2-1)/2$ and its null distribution is consistent with that in \cite{PV16}.

\subsection{Simulation experiments}

Simulations are carried out to evaluate the performance of proposed estimation and test for discrete PSDs in \eqref{psd}.
Samples of $(z_{ij})$ are drawn from $N(0,1)$ and all statistics are calculated from 10,000 independent replications.

The estimation procedure are conducted for two PSDs, Models 1 and 2: Model 1 is of  order 2 with the dimension to sample size ratio $c=2$ and Model 2 is of order 3 with the ratio $c=1/4$.
\begin{itemize}
	\item Model 1: $H_1=0.5\delta_{0.5}+0.5\delta_{1.5}$ and $c=2$.
	\item Model 2: $H_2=0.3\delta_{0.2}+0.4\delta_1+0.3\delta_{1.8}$ and $c=1/4$.
\end{itemize}
The sample size is $n=100, 200, 400$ for Model 1 and $n=400,800,1600$ for Model 2, respectively.
In addition to empirical means and standard deviations of all estimators, we also calculate 95\% confidence intervals for all parameters and report their coverage probabilities. Results are collected in Tables \ref{table1} and \ref{table2}, which clearly demonstrate the consistency of all estimators as the sample size $n$ become large.

\begin{table*}
	\caption{\label{table1} Estimation for Model 1 with sample size $n$= 100,200,400 and  $c=2$. The number of independent replications is 10,000 and the nominal coverage probability (C.\ P.) is fixed at $95\%$.}
	\centering
			\begin{tabular}{*{12}{c}}
				\hline
				$\theta$        &   \multicolumn{3}{c}{$n=100$} &&  \multicolumn{3}{c}{$n=200$} &&\multicolumn{3}{c}{$n=400$} \\
				&   Mean &St. D.& C. P.&&   Mean &St. D.&C. P.&&  Mean &St. D.&C. P. \\
				\cline{2-4} \cline{6-8}\cline{10-12}
				$a_1=0.5$&0.4839&0.1145&0.9375&&0.4960&0.0550&0.9491&&0.5000&0.0269&0.9486\\
				$w_1=0.5$&0.4915&0.1135&0.9137&&0.4968&0.0588&0.9423&&0.4997&0.0292&0.9488\\
				$a_2=1.5$&1.5030&0.1330&0.9288&&1.4990&0.0668&0.9426&&1.4998&0.0329&0.9487\\
				$w_2=0.5$&0.5085&0.1135&0.9137&&0.5032&0.0588&0.9423&&0.5003&0.0292&0.9488\\
				\hline
		\end{tabular}
\end{table*}

\begin{table*}
	\caption{\label{table2}Estimation for Model 2 with sample size $n$= 400,800,1600 and  $c=1/4$. The number of independent replications is 10,000 and the nominal coverage probability (C.\ P.) is fixed at $95\%$.}
	\centering
\begin{tabular}{lcccccccccccc}
				\hline
				$\theta$        &   \multicolumn{3}{c}{$n=400$} &&  \multicolumn{3}{c}{$n=800$} &&\multicolumn{3}{c}{$n=1600$} \\
				&   Mean &St. D.& C. P.&&   Mean &St. D.&C. P.&&  Mean &St. D.&C. P. \\
				\cline{2-4} \cline{6-8}\cline{10-12}
				$a_1=0.2$&0.1887&0.0429&0.9227&&0.1988&0.0147&0.9358&&0.2003&0.0071&0.9367\\
				$w_1=0.3$&0.2824&0.0447&0.9403&&0.2956&0.0184&0.9525&&0.2990&0.0090&0.9483\\
				$a_2=1.0$&0.9960&0.1347&0.9345&&0.9924&0.0661&0.9486&&0.9991&0.0337&0.9433\\
				$w_2=0.4$&0.4064&0.0373&0.9453&&0.4012&0.0209&0.9239&&0.4002&0.0110&0.9351\\ 	$a_3=1.8$&1.7824&0.0856&0.9236&&1.7919&0.0440&0.9413&&1.7960&0.0227&0.9392\\
				$w_3=0.3$&0.3113&0.0696&0.9221&&0.3031&0.0365&0.9429&&0.3008&0.0189&0.9420\\ 		
				\hline
		\end{tabular}
\end{table*}

Next we examine the test for the order of a PSD. Two models are employed for this experiment:
\begin{itemize}
	\item Model 3: $H_3=0.5\delta_{1-x}+0.5\delta_{1+x}$,
	\item Model 4: $H_4=0.25\delta_{0.5-x}+0.25\delta_{0.5+x}+0.25\delta_{1.5-x}+0.25\delta_{1.5+x}$,
\end{itemize}
where the parameter $x\in[0, 0.5)$ represents the distance between the null and alternative hypotheses. In particular, Model 3 is used for testing $H_0: d\leq 1$ (sphericity test) with $x$ ranging from 0 to 0.2 by a step 0.18 and Model 4 is for testing $H_0:d\leq 2$ with $x$ ranging from 0 to 0.45 by a step 0.05.
The sample size is taken as $n=400$, the dimension-sample size ratio is $c=1/2,1, 2$, and the significance level is fixed at $\alpha=0.05$. Results summarized in Table \ref{table3} show that the proposed test has accurate empirical size and its power tends to 1 as the parameter $x$ increases under the two models. Different from the sphericity test, the power for Model 2 declines significantly as the ratio $c$ increases. This phenomenon is consistent with that based on SCM depicted in \cite{Q16}.

\begin{table*}
	\caption{\label{table3}Empirical size and power of $T_n$ in percentage under Model 3 and Model 4 with the sample size $n=400$. The number of independent replications is 10,000 and the nominal significance level is $0.05$.}
	\centering
	\begin{tabular}{lcccccccccccc}
				\hline
				\multicolumn{11}{c}{$H_0: d\leq 1$ under Model 3}\\
				\hline
				$x$&0&0.02&0.04&0.06&0.08&0.10&0.12&0.14&0.16&0.18\\
				$c=\frac{1}{2}$&5.24&5.81&9.13&17.91&34.86&62.30&87.31&98.01&99.90&100\\
				$c=1$&5.33&5.92&8.43&18.09&35.62&63.12&88.14&98.69&99.96&100\\
				$c=2$&4.76&6.39&9.69&17.39&35.23&63.57&88.15&98.67&99.97&100\\ 	
				\hline
				\multicolumn{11}{c}{$H_0: d\leq 2$ under Model 4}\\
				\hline
				$x$&0&0.05&0.10&0.15&0.20&0.25&0.30&0.35&0.40&0.45\\
				$c=\frac{1}{2}$&4.75&7.19&17.49&43.96&79.28&97.06&99.87&100&100&100\\
				$c=1$&5.05&6.31&12.22&26.78&53.74&80.74&95.07&99.52&99.97&100\\
				$c=2$&4.88&5.65&8.56&16.33&30.09&49.17&71.60&86.54&95.20&98.61\\	
				\hline
		\end{tabular}
\end{table*}

\section{Proofs}
\label{sec:proofs}

\subsection{Some key lemmas}

We present three lemmas which form the core
basis for the proofs of Theorems \ref{lsd} and \ref{clt}.

\begin{lemma}\label{mom-y}
	Let $\bx=(x_1,\ldots,x_p)'\sim N_p(0, \T)$ where $\T=diag(\sigma_1^2,\ldots,\sigma_p^2)$ is a diagonal matrix with the spectral norm $||\T||$ bounded. Write $r_k=\sum_{i=1}^p\sigma_i^{2k}/p$, $k=1,2$. Then we have for $1\leq i\neq j\leq p$,
	\begin{eqnarray*}
		\E\left(\frac{x_i^2}{\sum_{i=1}^p x_i^2/p}\right)
		&=&\frac{\sigma_i^2}{r_{1}}+\frac{2\sigma_i^2 r_{2}-2\sigma_i^4 r_{1}}{p r_{1}^3}
		+o\bigg(\frac{1}{p}\bigg),\\
		\E\left(\frac{x_i^2x_j^2}{(\sum_{i=1}^p x_i^2/p)^2}\right)
		&=&\frac{\sigma_i^2\sigma_j^2}{ r_{1}^2}+\frac{6\sigma_i^2\sigma_j^2 r_{2}-4\sigma_i^2\sigma_j^2(\sigma_i^2+\sigma_j^2) r_{1}}{p r_{1}^4}
		+o\bigg(\frac{1}{p}\bigg)\\
		\E\left(\frac{x_i^4}{(\sum_{i=1}^p x_i^2/p)^2}\right)
		&=&\frac{3\sigma_i^4}{ r_{1}^2}+\frac{18\sigma_i^4 r_{2}-24\sigma_i^6 r_{1}}{p r_{1}^4}
		+o\bigg(\frac{1}{p}\bigg).
	\end{eqnarray*}
\end{lemma}

\begin{proof}
As three expectations can be evaluated through a similar way, we only present the details for the second one as an illustration.
Replacing the denominator of the quantity inside the expectation by $ r_1^2$ and making their difference yields
\begin{eqnarray}
\frac{x_i^2x_j^2}{(\sum_{i=1}^p x_i^2/p)^2}-\frac{x_i^2x_j^2}{ r_1^2}
&=&\frac{x_i^2x_j^2\big[ p^2r_1^2-(\sum_{i=1}^p x_i^2)^2\big]}{p^2 r_1^4}+\frac{x_i^2x_j^2\big[ p^2r_1^2-(\sum_{i=1}^p x_i^2)^2\big]^2}{p^4 r_1^6}+o_p\bigg(\frac{1}{p}\bigg)\nonumber\\
&:=&\frac{A_p}{ r_1^4}+\frac{B_p}{ r_1^6}+o_p\bigg(\frac{1}{p}\bigg),\label{diff}
\end{eqnarray}
where
\begin{eqnarray*}
	A_p
	&=&\frac{x_i^2x_j^2}{p^2}\left[\bigg(\sum_{k\neq i,j} \sigma_k^2\bigg)^2-\bigg(\sum_{k\neq i,j} x_k^2\bigg)^2+2\left[(\sigma_i^2+\sigma_j^2)-(x_i^2+x_j^2)\right]\sum_{k\neq i,j} x_k^2\right]+o_p\bigg(\frac{1}{p}\bigg),\\
	B_p&=&\frac{x_i^2x_j^2}{p^4}\left[\bigg(\sum_{k\neq i,j} \sigma_k^2\bigg)^2-\bigg(\sum_{k\neq i,j} x_k^2\bigg)^2\right]^2+o_p\bigg(\frac{1}{p}\bigg).
\end{eqnarray*}
Taking expectations of $A_p$ and $B_p$, we get
\begin{eqnarray*}
	\E(A_p)
	&=&-\frac{2\sigma_i^2\sigma_j^2 r_{2}+4\sigma_i^2\sigma_j^2(\sigma_i^2+\sigma_j^2) r_{1}}{p}+o\bigg(\frac{1}{p}\bigg),\\
	\E(B_p)&=&\frac{8\sigma_i^2\sigma_j^2 r_{2} r_{1}^2}{p}+o\bigg(\frac{1}{p}\bigg),
\end{eqnarray*}
which combined with \eqref{diff} gives
\begin{eqnarray*}
	\E\left(\frac{x_i^2x_j^2}{(\sum_{i=1}^p x_i^2/p)^2}\right)
	&=&\frac{\E(x_i^2x_j^2)}{ r_1^2}+\frac{\E(A_p)}{ r_1^4}+\frac{\E(B_p)}{ r_1^6}+o\bigg(\frac{1}{p}\bigg)\\
	&=&\frac{\sigma_i^2\sigma_j^2}{ r_{1}^2}+\frac{6\sigma_i^2\sigma_j^2 r_{2}-4\sigma_i^2\sigma_j^2(\sigma_i^2+\sigma_j^2) r_{1}}{p r_{1}^4}
	+o\bigg(\frac{1}{p}\bigg).
\end{eqnarray*}
\end{proof}

\begin{lemma}\label{double-e}
	Let $\y=\sqrt{p}\bx/||\bx||$ where $\bx$ is as defined in Lemma \ref{mom-y} such that $\E(\y\y')=\Sigma$. For any $p\times p$ complex matrices $\C$ and $\tilde \C$ with bounded spectral norms,
	\begin{eqnarray*}
		&&\E\big(\y'\C\y-\tr\Sigma \C\big)\big(\y'\tilde \C\y-\tr\Sigma\tilde \C\big)\\
		&=&\tr\Sigma \C\Sigma \tilde \C'+\tr\Sigma \C\Sigma \tilde \C
		+\frac{2}{p}\bigg( \gamma_2\tr\Sigma \C\tr\Sigma \tilde \C-\tr\Sigma^2\C\tr\Sigma\tilde \C-\tr\Sigma \C\tr \Sigma^2\tilde \C\bigg)+o(p),
	\end{eqnarray*}
	where $ \gamma_2=\tr\Sigma^2/p$.
\end{lemma}

\begin{proof}
By symmetry, $\E(y_i^3y_j)=\E(y_i^2y_jy_k)=\E(y_iy_jy_ky_l)=0$ for $1\leq i\neq j\neq k\neq l\leq p$. Write $\C=(C_{ij})$ and $\tilde \C=(\tilde C_{ij})$, we thus get
\begin{equation}\label{eyy}
\E(\y' \C\y)(\y'\tilde \C\y)=\sum_{i=1}^pC_{ii}\tilde C_{ii}\E(y_i^4)+\sum_{i\neq j}(C_{ii}\tilde C_{jj}+C_{ij}\tilde C_{ij}+C_{ij}\tilde C_{ji})\E(y_i^2y_j^2).
\end{equation}
From Lemma 1, we have
\begin{eqnarray*}
	\sum_{i\neq j}C_{ii}\tilde C_{jj}\E(y_i^2y_j^2)
	&=&\frac{\tr\T \C\tr\T \tilde \C}{r_{1}^2}+\frac{6r_2\tr\T \C\tr\T \tilde \C}{pr_1^4}-\frac{4r_{1}\left(\tr\T^2\C\tr\T\tilde \C+\tr\T \C\tr \T^2\tilde \C\right)}{pr_{1}^4}\\
	&&-\frac{1}{3}\sum_{i=1}^pC_{ii}\tilde C_{ii}\E(y_i^4)+o(p),\\
	\sum_{i\neq j}C_{ij}\tilde C_{ij}\E(y_i^2y_j^2)
	&=&\frac{\tr\T \C\T \tilde \C'}{r_{1}^2}-\frac{1}{3}\sum_{i=1}^pC_{ii}\tilde C_{ii}\E(y_i^4)+o(p),\\
	\sum_{i\neq j}C_{ij}\tilde C_{ji}\E(y_i^2y_j^2)
	&=&\frac{\tr\T \C\T \tilde \C}{r_{1}^2}-\frac{1}{3}\sum_{i=1}^pC_{ii}\tilde C_{ii}\E(y_i^4)+o(p).
\end{eqnarray*}
From the above quantities and \eqref{eyy}, we obtain
\begin{eqnarray*}
	\E(\y' \C\y)(\y'\tilde \C\y)&=&\frac{\tr\T \C\tr\T \tilde \C+\tr\T \C\T \tilde \C'+\tr\T \C\T \tilde \C}{r_{1}^2}\\
	&&+\frac{6r_2\tr\T \C\tr\T \tilde \C-4r_{1}\left(\tr\T^2\C\tr\T\tilde \C+\tr\T \C\tr \T^2\tilde \C\right)}{pr_{1}^4}+o(p).
\end{eqnarray*}
On the other hand, from the first conclusion of Lemma 1,  one may derive that
\begin{equation}\label{sts}
\tr \Sigma \bM=\frac{\tr \T \bM}{r_1}+\frac{2r_2\tr\T \bM-2r_1\tr \T^2 \bM}{pr_1^3}+o(1)=\frac{\tr \T \bM}{r_1}+O(1),
\end{equation}
for any $p\times p$ matrix $\bM$ with bounded spectral norm,	which implies
\begin{eqnarray*}
	\tr \Sigma \C\tr \Sigma \tilde \C&=&\left(\frac{\tr \T \C}{r_1}+\frac{2r_2\tr\T \C-2r_1\tr \T^2 \C}{pr_1^3}\right)
	\left(\frac{\tr \T \tilde \C}{r_1}+\frac{2r_2\tr\T \tilde \C-2r_1\tr \T^2 \tilde \C}{pr_1^3}\right)+o(p)\\
	&=&\frac{\tr \T \C\tr \T \tilde \C}{r_1^2}+\frac{4r_2\tr \T \C\tr\T \tilde \C-2r_1(\tr\T \C\tr \T^2 \tilde \C+\tr\T \tilde \C\tr \T^2 \C)}{pr_1^4}+o(p).
\end{eqnarray*}
Therefore,
\begin{eqnarray*}
	E\big(\y'\C\y-\tr\Sigma \C\big)\big(\y'\tilde \C\y-\tr\Sigma\tilde \C\big)&=&\E(\y' \C\y)(\y'\tilde \C\y)-\tr \Sigma \C\tr \Sigma \tilde \C\\
	&=&\frac{2r_2\tr\T \C\tr\T \tilde \C-2r_{1}\left(\tr\T^2\C\tr\T\tilde \C+\tr\T \C\tr \T^2\tilde \C\right)}{pr_{1}^4}\\
	&&+\frac{\tr\T \C\T \tilde \C'+\tr\T \C\T \tilde \C}{r_{1}^2}+o(p).
\end{eqnarray*}
Finally, from \eqref{sts}, we may replace $\T$ with $r_1\Sigma$ and replace $r_2/r_1^2$ with $\tr(\Sigma^2)/p$  in the above expression and then obtain the result of the Lemma.
\end{proof}

Let $v_0 > 0$ be arbitrary,  $x_r$ any number greater than $\limsup_{p\rightarrow\infty}\lambda_{\max}^{\Sigma}(1+\sqrt{c})^2$, and $x_l$ any negative number if $\liminf_{p\rightarrow\infty}\lambda_{\min}^{\Sigma}(1-\sqrt{c})^2I_{(0,1)}(c)=0$, otherwise choose
$x_l\in(0,\liminf_{p\rightarrow\infty}\lambda_{\min}^{\Sigma}(1-\sqrt{c})^2)$. Define a contour $\mathcal C$ as
\begin{eqnarray}\label{cont}
\mathcal C= \{x\pm iv_0: x\in[x_l,x_r]\}\cup \{x+iv: x\in \{x_r,x_l\}, v\in[-v_0,v_0]\}.
\end{eqnarray}

Let $m_{0}(z)$ and $\um_0(z)$ be the Stieltjes transforms of $F^{c_n,  H_p}$ and $c_nF^{c_n,  H_p} + (1-c_n)\de_0$.
Our next aim is to study the fluctuation of the random process
$$
M_{n}(z)=p[m_{n}(z)-m_{0}(z)]=n[\um_{n}(z)-\um_{0}(z)],\quad z\in\mathcal C.
$$
For this, we define a truncated version $\widehat M_{n}(z)$ of $M_{n}(z)$ as
\begin{eqnarray}
\widehat M_{n}(z)=
\begin{cases}
M_{n}(z)&z\in \mathcal C_{n},\\
M_{n}(x+in^{-1}\varepsilon_n)& x\in \{x_l,x_r\}\quad\text{and}\quad v\in[0, n^{-1}\varepsilon_n],\\
M_{n}(x-in^{-1}\varepsilon_n)& x\in \{x_l,x_r\}\quad\text{and}\quad v\in[-n^{-1}\varepsilon_n, 0],
\end{cases}
\end{eqnarray}
where $ \mathcal C_{n}= \{x\pm iv_0: x\in[x_l,x_r]\}\cup \{x\pm iv: x\in\{x_l, x_r\}, v\in[n^{-1}\varepsilon_n, v_0]\}$ and the sequence $(\varepsilon_n)$ decreasing to zero satisfying  $\varepsilon_n>n^{-a}$ for some $a\in(0,1)$.

\begin{lemma}\label{clt-mn}
	Under Assumptions (a)-(c), the random process $\widehat M_n(\cdot)$ converges weakly to a two-dimensional Gaussian process $M(\cdot)$ satisfying for $z, z_1, z_2\in \mathcal C$,
	\begin{align}
	{\E}M(z)
	&=\int\frac{c(\um'(z)t)^2d H(t)}{\um(z)(1+\um(z) t)^3}+2c\um(z)\um'(z)\bigg[\int\frac{ \gamma_2t-t^2d H(t)}{1+\um(z) t}\int\frac{td H(t)}{(1+\um(z) t)^2}\nonumber\\
	&-\int\frac{td H(t)}{1+\um(z) t}\int\frac{t^2d H(t)}{(1+\um(z) t)^2}\bigg]\label{mean}
	\end{align}
	and covariance function
	\begin{align}
	&\Cov(M(z_1),M(z_2))\nonumber\\
	=&\frac{2\um'(z_1)\um'(z_2)}{(\um(z_1)-\um(z_2))^2}-\frac{2}{(z_1-z_2)^2}+\frac{2 \gamma_2}{c}\left(\um(z_1)+z_1\um'(z_1)\right)\left(\um(z_2)+z_2\um'(z_2)\right)\nonumber\\
	&-\frac{2}{c}\left(\frac{\um'(z_1)}{\um^2(z_1)}-1\right)\left(\um(z_2)+z_2\um'(z_2)\right)-\frac{2}{c}\left(\frac{\um'(z_2)}{\um^2(z_2)}-1\right)\left(\um(z_1)+z_1\um'(z_1)\right).\label{clt-cov}
	\end{align}
\end{lemma}

\begin{proof}

Split $\widehat M_{n}(z)$ into two parts,  $\widehat M_{n}(z) = M_{n}^{(1)}(z) + M_{n}^{(2)}(z)$, where
$$M_{n}^{(1)}(z) = p[m_{n} (z) - \E m_{n} (z)]\quad \text{and} \quad M_{n}^{(2)}(z) = p[\E m_{n} (z)-m_{0}(z)].$$
Following the strategy in \cite{BS04}, we prove the convergence of $\widehat M_{n}(z)$ by three steps:
\begin{itemize}
	\item[] Step 1: Finite dimensional convergence of $M_{n}^{(1)}(z)$ in distribution;
	\item[] Step 2: Tightness of $M_{n}^{(1)}(z)$ on $\mathcal C_{n}$;
	\item[] Step 3: Convergence of $M_{n}^{(2)}(z)$.
\end{itemize}

Without loss of generality, we assume $\|\Sigma\|\leq1$ for all $p$.
Constants appearing in inequalities will be denoted by $K$  which may take different values from one expression to the next.

\textit{Step 1: Finite dimensional convergence of $M_{n}^{(1)}(z)$ in distribution.}
We show in this part, for any $w$ complex numbers
$z_{1},\ldots, z_w \in \mathcal C_{n}$, the random vector
\begin{equation}\label{rv}
\left[M_{n}^{(1)}(z_1),\ldots,M_{n}^{(1)}(z_w)\right]
\end{equation}
converges in distribution to a Gaussian vector. We begin with introducing some notation which will be frequently used in the sequel.
\begin{eqnarray*}
	&&\br_j=(1/\sqrt{n})\y_j,\quad \D(z)=\B_{n}-zI,\\
	&& \D_{j}(z)=\D (z)-\br_j\br'_j,\quad \D_{ij}(z)=\D (z)-\br_i\br_i'-\br_j\br'_j,\\
	&&\varepsilon_{j}(z)=\br_j'\D_{j}^{-1}(z)\br_j-\frac{1}{n}\tr\Sigma\D_{j}^{-1}(z),\quad \delta_{j}(z)=\br_j'\D_{j}^{-2}(z)\br_j-\frac{1}{n}\tr\Sigma \D_{j}^{-2}(z),\\
	&&\beta_{j}(z)=\frac{1}{1+\br_j'\D_{j}^{-1}(z)\br_j},\quad \bar \beta_{j}(z)=\frac{1}{1+n^{-1}\tr\Sigma \D_{j}^{-1}(z)},\\
	&& b_{n}(z)=\frac{1}{1+n^{-1}\E\tr\Sigma \D_{j}^{-1}(z)}.
\end{eqnarray*}
Note that, for any $z=u+iv\in \mathbb C^+$, the last three quantities are bounded in absolute value by $|z|/v$.

Let $\E_0(\cdot)$ denote expectation and $\E_j(\cdot)$ denote conditional expectation with respect to the $\sigma$-field generated by $\br_1,\ldots,\br_j$.
From the martingale decomposition and the identity
\begin{equation}\label{D-D}
\D ^{-1}(z)-\D_{j}^{-1}(z)=-\D_{j}^{-1}(z)\br_j\br_j'\D_{j}^{-1}(z)\beta_{j}(z),
\end{equation}
we have
\begin{eqnarray}
M_{n}^{(1)}(z)&=&\tr(\D ^{-1}(z)-\E\D ^{-1}(z))\nonumber\\
&=&\sum_{j=1}^n\tr \E_j\D ^{-1}(z)-\tr \E_{j-1}\D ^{-1}(z)\nonumber\\
&=&\sum_{j=1}^n\tr \E_j[\D ^{-1}(z)-\D_{j}^{-1}(z)]-\tr \E_{j-1}[\D ^{-1}(z)-\D_{j}^{-1}(z)]\nonumber\\
&=&-\sum_{j=1}^n(\E_j-\E_{j-1})\beta_{j}(z)\br_j'\D_{j}^{-2}\br_j.\label{md}
\end{eqnarray}
Writing $\beta_{j}(z)=\bar\beta_{j}(z)-\bar\beta_{j}(z)\beta_{j}(z)\varepsilon_{j}(z)=\bar \beta_{j}(z)-\bar\beta_{j}^2\varepsilon_{j}(z)+\bar\beta_{j}^2(z)\beta_{j}(z)\varepsilon_{j}^2(z)$, we have
\begin{eqnarray*}
	&&(\E_j-\E_{j-1})\beta_{j}(z)\br_j'\D_{j}^{-2}\br_j\\
	&=&(\E_j-\E_{j-1})\left(\bar\beta_{j}(z)\delta_{j}(z)-\bar\beta_{j}^2(z)\varepsilon_{j}(z)\br_j'\D_{j}^{-2}(z)\br_j+\bar\beta_{j}^2(z)\beta_{j}(z)\varepsilon_{j}^2(z)\br_j'\D_{j}^{-2}(z)\br_j\right)\\
	&=&\frac{d}{dz}\E_j\bar\beta_{j}(z)\varepsilon_{j}(z)-(\E_j-\E_{j-1})\bar\beta_{j}^2(z)\left(\varepsilon_{j}(z)\delta_{j}(z)-\beta_{j}(z)\varepsilon_{j}^2(z)\br_j'\D_{j}^{-2}(z)\br_j\right).
\end{eqnarray*}
Note that
\begin{eqnarray*}
	\E\bigg|\sum_{j=1}^n(\E_j-\E_{j-1})\bar\beta_{j}^2(z)\varepsilon_{j}(z)\delta_{j}(z)\bigg|^2&=&\sum_{j=1}^n\E|(\E_j-\E_{j-1})\bar\beta_{j}^2(z)\varepsilon_{j}(z)\delta_{j}(z)|^2\\
	&\leq&4\sum_{j=1}^n\E|\bar\beta_{j}^2(z)\varepsilon_{j}(z)\delta_{j}(z)|^2\\
	&\leq&\frac{4|z|^4}{v^4}\sum_{j=1}^nE^\frac{1}{2}|\varepsilon_{j}(z)|^4E^\frac{1}{2}|\delta_{j}(z)|^4
\end{eqnarray*}
which is $o(1)$ from Lemma \ref{ineq}. Similarly, $\E|\sum_{j=1}^n(\E_j-\E_{j-1})\bar\beta_{j}^2(z)\beta_{j}(z)\varepsilon_{j}^2(z)\br_j'\D_{j}^{-2}(z)\br_j|^2 =o(1)$.
Thus we get
$$
\sum_{j=1}^n(\E_j-\E_{j-1})\bar\beta_{j}^2(z)\left(\varepsilon_{j}(z)\delta_{j}(z)+\beta_{j}(z)\varepsilon_{j}^2(z)\br_j'\D_{j}^{-2}(z)\br_j\right)=o_p(1)
$$
which implies that we need only to consider the limiting distribution of
$$
-\frac{d}{dz}\sum_{j=1}^n \E_j\bar\beta_{j}(z)\varepsilon_{j}(z)=-\frac{d}{dz}\sum_{j=1}^n (\E_j-\E_{j-1})\bar\beta_{j}(z)\varepsilon_{j}(z)
$$
in finite dimensional situations.
For any $\epsilon>0$,
\begin{eqnarray*}
	&&  \sum_{j=1}^n\E\left|\E_j\frac{d}{dz}\varepsilon_{j}(z)\bar\beta_{j}(z)\right|^2 I_{\left(\left|\E_j\frac{d}{dz}\varepsilon_{j}(z)\bar\beta_{j}(z)\right|\geq\epsilon\right)}\\
	&\leq&\frac{1}{\epsilon^2}\sum_{j=1}^n \E\Big|\E_j\frac{d}{dz}\varepsilon_{j}(z)\bar\beta_{j}(z)\Big|^4
	\leq\frac{K}{\varepsilon^2}\sum_{j=1}^n
	\left(\frac{|z|^4\E|\delta_{j}(z)|^4}{v^4}+\frac{|z|^8p^4\E|\varepsilon_{j}(z)|^4}{v^{16}n^4}\right)
\end{eqnarray*}
which tends to zero according to Lemma \ref{ineq} and thus verifies the Lyapunov condition.
Therefore, from the martingale CLT (Lemma \ref{MCLT}),
the random vector in \eqref{rv} will tend to a Gaussian vector
$(M^{(1)}(z_1),\ldots, M^{(1)}(z_w))$ with covariance function
\begin{equation}\label{lcov}
Cov(M^{(1)}(z_1),M^{(1)}(z_2))=\lim_{n\to\infty}\sum_{j=1}^n \frac{\partial^2}{\partial z_1\partial z_2}\E_{j-1}\left(\E_j\varepsilon_{j}(z_1)\bar\beta_{j}(z_1)\cdot
\E_j\varepsilon_{j}(z_2)\bar\beta_{j}(z_2)\right),
\end{equation}
provided this limit exits. By the same arguments in page 571 of \cite{BS04}, it is sufficient to show that
\begin{equation}\label{lcov-1}
\sum_{j=1}^n\E_{j-1}\prod_{k=1}^2\E_j\bar \beta_j(z_k)\varepsilon_j(z_k)
\end{equation}
converges in probability. Since
\begin{eqnarray}
E|\bar\beta_j(z)-b_n(z)|^2&=&|b_n(z)|^2n^{-2}\E|\bar\beta_1(z)(\tr\Sigma \D_1^{-1}(z)-E\tr\Sigma \D_1^{-1}(z))|^2\nonumber\\
&\leq&\frac{|z|^4}{n^2v^4}\E\bigg|\sum_{k=2}^n(\E_k-\E_{k-1})\tr(\D_1^{-1}(z)-\D_{1k}^{-1}(z))\bigg|^2\nonumber\\
&\leq&\frac{K|z|^4}{n^2v^4}\E\sum_{k=2}^n\big|\tr(\D_1^{-1}(z)-\D_{1k}^{-1}(z))\big|^2\nonumber\\
&\leq&\frac{K|z|^4}{v^6n},\label{beta-b}
\end{eqnarray}
where the last inequality is from
\begin{equation}\label{d-d}
|\tr(\D ^{-1}(z)-\D_{j}^{-1}(z))\bM|\leq\frac{||\bM||}{v},
\end{equation}
for any $p\times p$ matrix $\bM$, see Lemma 2.6 in \cite{SB95}.
Moreover,
from the definition of $\um_0(z)$ and discussions in Page 439 in \cite{BZ08},
we also have
\begin{equation}\label{bnz}
b_{n}(z)+z\um_0(z)\rightarrow 0.
\end{equation}
It is hence sufficient to
study the convergence of
\begin{equation}\label{cov-term-2}
z_1z_2\um_0(z_1)\um_0(z_2)\sum_{j=1}^n\E_{j-1}\left(\E_j\varepsilon_j(z_1)\E_j\varepsilon_j(z_2)\right),
\end{equation}
whose second mixed partial derivative yields the limit of \eqref{lcov}.
From Lemma 2, we know that
\begin{equation}\label{cov-term-3}
\eqref{cov-term-2}= 2(T_1+ \gamma_{n2}T_2-T_3-T_4),
\end{equation}
where
\begin{eqnarray*}
	T_1&=&\frac{z_1z_2\um_0(z_1)\um_0(z_2)}{n^2}\sum_{j=1}^n\tr\left[\E_j\Sigma \D_{j}^{-1}(z_1)\E_j(\Sigma \D_{j}^{-1}(z_2))\right],\\
	T_2&=&\frac{z_1z_2\um_0(z_1)\um_0(z_2)}{pn^2}\sum_{j=1}^n\tr\left[\E_j\Sigma \D_{j}^{-1}(z_1)\right]\tr\left[\E_j\Sigma \D_{j}^{-1}(z_2)\right],\\
	T_3&=&\frac{z_1z_2\um_0(z_1)\um_0(z_2)}{pn^2}\sum_{j=1}^n\tr\left[\E_j\Sigma^2 \D_{j}^{-1}(z_1)\right]\tr\left[\E_j\Sigma \D_{j}^{-1}(z_2)\right],\\
	T_4&=&\frac{z_1z_2\um_0(z_1)\um_0(z_2)}{pn^2}\sum_{j=1}^n\tr\left[\E_j\Sigma \D_{j}^{-1}(z_1)\right]\tr\left[\E_j\Sigma^2 \D_{j}^{-1}(z_2)\right].
\end{eqnarray*}

Now we consider the limit of $T_1$. Let
\begin{eqnarray*}
	\beta_{ij}(z)=\frac{1}{1+\br_i'\D_{ij}^{-1}(z)\br_i},\quad b_{1}(z)=\frac{1}{1+n^{-1}\E\tr \Sigma \D_{12}^{-1}(z)},\quad \bL(z)=zI-\frac{n-1}{n}b_1(z)\Sigma.
\end{eqnarray*}
Note that
\begin{equation}\label{z-b}
||\bL(z)||^{-1}\leq \frac{|b_1^{-1}(z)|}{\Im(zb_1^{-1}(z))}\leq\frac{|b_1^{-1}(z)|}{\Im(z)}\leq\frac{1+p/(nv)}{v}.
\end{equation}
From the equality $\br_i'\D_j^{-1}(z)=\beta_{ij}(z)\br_i'\D_{ij}^{-1}(z)$, we get
\begin{eqnarray}
\D_j^{-1}(z)+\bL^{-1}(z)&=&\bL^{-1}(z)\left(\D_j(z)+\bL(z)\right)\D_j^{-1}(z)\nonumber\\
&=&\bL^{-1}(z)\left(\sum_{i\neq j}\br_i\br_i'-\frac{n-1}{n}b_1(z)\Sigma\right)\D_j^{-1}(z)\nonumber\\
&=&\bL^{-1}(z)\left(\sum_{i\neq j}\br_i\beta_{ij}(z)\br_i'\D_{ij}^{-1}(z)-\frac{n-1}{n}b_1(z)\Sigma \D_j^{-1}(z)\right)\nonumber\\
&=&b_1(z)\bR_1(z)+\bR_2(z)+\bR_3(z),\label{dj-1}
\end{eqnarray}
where
\begin{eqnarray*}
	\bR_1(z)&=&\sum_{i\neq j}\bL^{-1}(z)(\br_i\br_i'-n^{-1}\Sigma)\D_{ij}^{-1}(z),\\
	\bR_2(z)&=&\sum_{i\neq j}(\beta_{ij}(z)-b_1(z))\bL^{-1}(z)\br_i\br_i'\D_{ij}^{-1}(z),\\
	\bR_3(z)&=&n^{-1}b_1(z)\bL^{-1}(z)\Sigma \sum_{i\neq j}\left(\D_{ij}^{-1}(z)-\D_{j}^{-1}(z)\right).
\end{eqnarray*}

For any $p\times p$ matrix $\bM$, let $|||\bM|||$ denote a non-random upper bound for the spectral norm of $\bM$.
From Lemma \ref{ineq}, \eqref{d-d}, and \eqref{z-b}, we get
\begin{eqnarray}
\E|\tr \bR_1(z)\bM|&\leq& n\E^{1/2}|\br_1'\D_{12}^{-1}(z)\bM\bL^{-1}(z)\br_1-n^{-1}\tr\Sigma \D_{12}^{-1}(z)\bM\bL^{-1}(z)|^2\nonumber\\
&\leq&n^{1/2}K||\bM||\frac{(1+p/(nv))}{v^2}\label{AM},\\
\E|\tr \bR_2(z)\bM|&\leq& n\E^{1/2}(|\beta_{12}(z)-b_1(z)|^2)\E^{1/2}\bigg|\br_1'\D_{12}^{-1}\bM\bL^{-1}(z)\br_1\bigg|^2\nonumber\\
&\leq&n^{1/2} K |||\bM|||\frac{|z|^2(1+p/(nv))}{v^5},\label{BM}\\
|\tr \bR_3(z)M|&\leq&|||\bM|||\frac{|z|(1+p/(nv))}{v^3},\label{CM}
\end{eqnarray}
where the matrix $M$ in the first two inequalities  is assumed nonrandom.

Using the equality \eqref{D-D} we write
\begin{equation}\label{R123}
\tr \E_j(\bR_1(z_1)) \Sigma \D_j^{-1}(z_2)\Sigma=R_{11}(z_1,z_2)+R_{12}(z_1,z_2)+R_{13}(z_1,z_2),
\end{equation}
where
\begin{eqnarray*}
	R_{11}(z_1,z_2)&=&-\sum_{i< j}\beta_{ij}(z_2)\br_i'\E_j(\D_{ij}^{-1}(z_1))\Sigma \D_{ij}^{-1}(z_2)\br_i\br_i'\D_{ij}^{-1}(z_2)\Sigma \bL^{-1}(z_1)\br_i,\\
	R_{12}(z_1,z_2)&=&-\tr\sum_{i<j}L^{-1}(z_1)n^{-1}\Sigma \E_j(\D_{ij}^{-1}(z_1))\Sigma(\D_j^{-1}(z_2)-\D_{ij}^{-1}(z_2))\Sigma,\\
	R_{13}(z_1,z_2)&=&\tr\sum_{i<j}L^{-1}(z_1)(\br_i\br_i'-n^{-1}\Sigma)\E_j(\D_{ij}^{-1}(z_1))\Sigma \D_{ij}^{-1}(z_2)\Sigma.
\end{eqnarray*}
From \eqref{d-d} and \eqref{z-b} we get $|R_{12}(z_1,z_2)|\leq (1+p/(nv))/v^3$ and $\E|R_{13}(z_1,z_2)|\leq n^{1/2}(1+p/(nv))/v^3$.
Using Lemma \ref{ineq} we have, for $i<j$,
\begin{eqnarray*}
	&&\E\bigg| \beta_{ij}(z_2)\br_i'\E_j(\D_{ij}^{-1}(z_1))\Sigma \D_{ij}^{-1}(z_2)\br_i\br_i'\D_{ij}^{-1}(z_2)\Sigma \bL^{-1}(z_1)\br_i\\
	&&-b_1(z_2)n^{-2}\tr \left(\E_j(\D_{ij}^{-1}(z_1))\Sigma \D_{ij}^{-1}(z_2)\Sigma\right) \tr\left( \D_{ij}^{-1}(z_2)\Sigma \bL^{-1}(z_1)\T\right)\bigg|\leq Kn^{-1/2},
\end{eqnarray*}
and by \eqref{d-d},
\begin{eqnarray*}
	&&\bigg|\tr \left(\E_j(\D_{ij}^{-1}(z_1))\Sigma \D_{ij}^{-1}(z_2)\Sigma\right) \tr \left(\D_{ij}^{-1}(z_2)\Sigma \bL^{-1}(z_1)\Sigma\right)\\
	&&-\tr\left( \E_j(\D_{j}^{-1}(z_1))\Sigma \D_{j}^{-1}(z_2)\Sigma \right)\tr\left( \D_{j}^{-1}(z_2)\Sigma \L^{-1}(z_1)\Sigma\right)\bigg|\leq Kn.
\end{eqnarray*}
These imply that
\begin{align}
&\E\bigg|R_{11}(z_1,z_2)+\frac{j-1}{n^2}b_1(z_2)\tr\left( \E_j(\D_{j}^{-1}(z_1))\Sigma \D_{j}^{-1}(z_2)\Sigma \right)\tr\left( \D_{j}^{-1}(z_2)\Sigma \bL^{-1}(z_1)\T\right)\bigg|\label{az12}\\
&\leq Kn^{1/2}.\nonumber
\end{align}
Therefore, from \eqref{dj-1}-\eqref{az12},
\begin{eqnarray*}
	&&\tr\left( \E_j(\D_{j}^{-1}(z_1))\Sigma \D_{j}^{-1}(z_2)\Sigma \right)\left(1+\frac{j-1}{n^2}b_1(z_1)b_1(z_2)\tr\left( \D_{j}^{-1}(z_2)\Sigma \bL^{-1}(z_1)\Sigma\right)\right)\\
	&=&-\tr \bL^{-1}(z_1)\Sigma \D_j^{-1}(z_2)\Sigma +R_{14}(z_1,z_2),
\end{eqnarray*}
where $\E|R_{14}(z_1,z_2)|\leq Kn^{1/2}$. From this and applying \eqref{dj-1}-\eqref{az12} again, we get
\begin{align}\label{R15}
&\tr\left( \E_j(\D_{j}^{-1}(z_1))\Sigma \D_{j}^{-1}(z_2)\Sigma \right)\left(1-\frac{j-1}{n^2}b_1(z_1)b_1(z_2)\tr\left( \bL^{-1}(z_2)\Sigma \bL^{-1}(z_1)\Sigma\right)\right)\\
&=\tr \bL^{-1}(z_1)\Sigma \bL^{-1}(z_2)\Sigma +R_{15}(z_1,z_2),\nonumber
\end{align}
where $\E|R_{15}(z_1,z_2)|\leq Kn^{1/2}$.

From \eqref{bnz} and \eqref{R15}, we obtain that
\begin{align}
&\tr\left( \E_j(\D_{j}^{-1}(z_1))\Sigma \D_{j}^{-1}(z_2)\Sigma \right)\left(1-\frac{j-1}{n^2}\um_{0}(z_1)\um_{0}(z_2)\right.\nonumber\\
&\times\tr\left( (I+\um_{0}(z_2)\Sigma)^{-1}\Sigma  (I+\um_{0}(z_1)\Sigma)^{-1}\Sigma\right)\bigg)\nonumber\\
&=\tr\left( \E_j(\D_{j}^{-1}(z_1))\Sigma \D_{j}^{-1}(z_2)\Sigma \right)\left(1-\frac{j-1}{n}\int\frac{c_n\um_{0}(z_1)\um_{0}(z_2)t^2d H_p(t)}{(1+t\um_{0}(z_1))(1+t\um_{0}(z_2))}\right)\nonumber\\
&=\frac{nc_n}{z_1z_2}\int\frac{t^2d H_p(t)}{(1+t\um_{0}(z_1))(1+t\um_{0}(z_2))} +R_{16}(z_1,z_2).\label{var-2}
\end{align}
Here $\E|R_{16}(z_1,z_2)|\leq Kn^{1/2}$. Letting
$$
a_n(z_1,z_2)=\int\frac{c_n\um_{0}(z_1)\um_{0}(z_2)t^2d H_p(t)}{(1+t\um_{0}(z_1))(1+t\um_{0}(z_2))},
$$
we get
\begin{eqnarray*}
	T_1=\frac{1}{n}\sum_{j=1}^n\frac{a_n(z_1,z_2)}{1-((j-1)/n)a_n(z_1,z_2)}+o_p(1)\xrightarrow{i.p.}\int_0^{a(z_1,z_2)}\frac{1}{1-z}dz,
\end{eqnarray*}
where
\begin{eqnarray*}
	a(z_1,z_2)=\int\frac{c\um(z_1)\um(z_2)t^2d H(t)}{(1+t\um(z_1))(1+t\um(z_2))}=1+\frac{\um(z_1)\um(z_2)(z_1-z_2)}{\um(z_2)-\um(z_1)}.
\end{eqnarray*}
Elementary calculations reveal that
\begin{equation}\label{p-t1}
\frac{\partial^2T_1}{\partial z_1\partial z_2}=\frac{\um'(z_1)\um'(z_2)}{(\um(z_1)-\um(z_2))^2}-\frac{1}{(z_1-z_2)^2}.
\end{equation}

Now we derive the limits of $T_2$, $T_3$, $T_4$ and their second mixed partial derivatives.  From \eqref{bnz}, \eqref{dj-1}-\eqref{CM}, it's easy to show that
\begin{eqnarray*}
	\tr \E_j\D_j^{-1}(z_1)\Sigma\tr \E_j\D_j^{-1}(z_2)\Sigma
	&=&\frac{p^2}{z_1z_2}\int\frac{td H_p(t)}{1+t\um_{0}(z_1)}\int\frac{td H_p(t)}{1+t\um_{0}(z_2)}+R_{17}(z_1,z_2),\\
	\tr \E_j\D_j^{-1}(z_1)\Sigma^2\tr \E_j\D_j^{-1}(z_2)\Sigma
	&=&\frac{p^2}{z_1z_2}\int\frac{t^2d H_p(t)}{1+t\um_{0}(z_1)}\int\frac{td H_p(t)}{1+t\um_{0}(z_2)}+R_{18}(z_1,z_2),
\end{eqnarray*}
where $\E|R_{17}(z_1,z_2)|\leq Kn$ and $\E|R_{18}(z_1,z_2)|\leq Kn$. We thus get
\begin{eqnarray*}
	T_2&=&c_n\int\frac{t\um_0(z_1)d H_p(t)}{1+t\um_{0}(z_1)}\int\frac{t\um_0(z_2)d H_p(t)}{1+t\um_{0}(z_2)}+o_p(1)\xrightarrow{i.p.}\frac{1}{c}\left(1+z_1\um(z_1)\right)\left(1+z_2\um(z_2)\right),\\
	T_3&=&c_n\int\frac{t^2\um_0(z_1)d H_p(t)}{1+t\um_{0}(z_1)}\int\frac{t\um_0(z_2)d H_p(t)}{1+t\um_{0}(z_2)}+o_p(1)\xrightarrow{i.p.}\int\frac{t^2\um(z_1)\left(1+z_2\um(z_2)\right)d H_p(t)}{1+t\um(z_1)},\\
	T_4&=&c_n\int\frac{t\um_0(z_1)d H_p(t)}{1+t^2\um_{0}(z_1)}\int\frac{t\um_0(z_2)d H_p(t)}{1+t\um_{0}(z_2)}+o_p(1)\xrightarrow{i.p.}\int\frac{t^2\um(z_2)\left(1+z_1\um(z_1)\right)d H_p(t)}{1+t\um(z_2)}.
\end{eqnarray*}
Their corresponding derivatives are
\begin{eqnarray}
\frac{\partial^2T_2}{\partial z_1\partial z_2}&=&\frac{1}{c}\left(\um(z_1)+z_1\um'(z_1)\right)\left(\um(z_2)+z_2\um'(z_2)\right),\\
\frac{\partial^2T_3}{\partial z_1\partial z_2}&=&\int\frac{t^2\um'(z_1)\left(\um(z_2)+z_2\um'(z_2)\right)d H_p(t)}{(1+t\um(z_1))^2}\nonumber\\
&=&\frac{1}{c}\left(\frac{\um'(z_1)}{\um^2(z_1)}-1\right)\left(\um(z_2)+z_2\um'(z_2)\right),\\
\frac{\partial^2T_4}{\partial z_1\partial z_2}&=&\int\frac{t^2\um'(z_2)\left(\um(z_1)+z_1\um'(z_1)\right)d H_p(t)}{(1+t\um(z_2))^2},\nonumber\\
&=&\frac{1}{c}\left(\frac{\um'(z_2)}{\um^2(z_2)}-1\right)\left(\um(z_1)+z_1\um'(z_1)\right),\label{p-t4}
\end{eqnarray}
respectively.

Collecting results in \eqref{cov-term-3}, \eqref{p-t1}-\eqref{p-t4}, we finally get the covariance function in the lemma.

\textit{Step 2: Tightness of $M_n^{(1)}(z)$.}
From the arguments in \cite{BS04}, the tightness of $M_n^{(1)}(z)$ can be established by verifying the moment condition:
\begin{equation}\label{tightness}
\sup_{n,z_1,z_2\in \mathcal C_n}\frac{\E|M_n^{(1)}(z_1)-M_n^{(1)}(z_2)|^2}{|z_1-z_2|^2}<\infty.
\end{equation}
We first claim that moments of $\D^{-1}(z)$, $\D^{-1}_j(z)$ and $\D^{-1}_{ij}(z)$ are all bounded in $n$ and $z\in \mathcal C_n$. Taking $\D^{-1}(z)$ for example,
it's clear that $\E||\D^{-1}(z)||^q<1/v_0^q$ for $z\in \mathcal C_u$. For $z\in \mathcal C_l\cup C_r$, applying Lemma \ref{lambda-bound} with suitably large $s$,
\begin{eqnarray*}
	\E||\D^{-1}(z)||^q&\leq& K_1+\frac{1}{v^q}P(||\B_n||>\eta_r\ \text{or}\ \lambda_{\min}^{\B_n}<\eta_l)\\
	&\leq& K_1+K_2n^q\varepsilon^{-q}n^{-s}\leq K,
\end{eqnarray*}
where the two constant $\eta_r$ and $\eta_l$ satisfy
$\limsup_{n,p\rightarrow\infty}\lambda_{\max}^{\Sigma}(1+\sqrt{c})^2<\eta_r<x_r$ and $x_l<\eta_l<\liminf_{n,p\rightarrow\infty}\lambda_{\min}^{\Sigma}I_{(0,1)}(c)(1-\sqrt{c})^2$.
Therefore for any positive $q$, we may assume that
\begin{equation}\label{D-bound}
\max \{ \E||\D^{-1}(z)||^q, \E||\D^{-1}_j(z)||^q, \E||\D^{-1}_{ij}(z)||^q\}\leq K_q.
\end{equation}
Using the above argument, we can extend the inequality in Lemma \ref{ineq} to
\begin{eqnarray}
\bigg|E\left[a(v)\prod_{l=1}^{q}\left(\y_1'\B_l(v)\y_1-\tr\Sigma \B_l(v)\right)\right]\bigg|
\leq Kp^{q/2},\label{ineq-ex}
\end{eqnarray}
where the matrices $\B_l(v)$ are independent of $\bu_1$ and
\begin{equation}\label{ab-bound}
\max\{|a(v)|, ||\B_l(v)|| \}\leq K\left[1+p^{s}I\left(||\B_n||\geq\eta_r \ \text{or}\ \lambda_{\min}^{\tilde \B}\leq\eta_l\right)\right]
\end{equation}
for some positive $s$, where $\tilde \B$ is $\B_n$ or $\B_n$ with some $\br_j$'s removed.
In applications of \eqref{ineq-ex},  $a(v)$ can be a product of factors of $\beta_1(z)$ or $\br_1'\D_1^{-1}(z_1)\D_1^{-1}(z_2)\br_1$ or similar terms.
It's easy to verify that these terms satisfy \eqref{ab-bound}, see pages 579 and 580 in \cite{BS04} for details.

Let
$$\gamma_j(z)=\br_j'\D_j^{-1}(z)\br_j-\frac{1}{n}\E\tr\Sigma \D_j^{-1}(z).$$ We first handle moments of $\gamma_j(z)$.
By a similar decomposition in \eqref{md}, we may get
\begin{eqnarray}
\E|\gamma_j(z)-\varepsilon_j(z)|^q&=&\E\bigg|\frac{1}{n}\sum_{i\neq j}(\E_i-\E_{i-1}) \beta_{ij}(z)\br_i'\D_{ij}^{-1}(z)\Sigma \D_{ij}^{-1}(z)\br_i\bigg|^q.\nonumber
\end{eqnarray}
Applying Lemma \ref{mi2} and the H\"{o}lder inequality to the above expression we then get, for even $q$,
\begin{eqnarray}
\E|\gamma_j(z)-\varepsilon_j(z)|^q
&\leq&\frac{K}{n^q}\E\left[\sum_{i\neq j}|(\E_i-\E_{i-1}) \beta_{ij}(z)\br_i'\D_{ij}^{-1}(z)\Sigma \D_{ij}^{-1}(z)\br_i\big|^2\right]^{q/2}\nonumber\\
&\leq&\frac{K}{n^{1+q/2}}\sum_{i\neq j}\E\big|(\E_i-\E_{i-1}) \beta_{ij}(z)\br_i'\D_{ij}^{-1}(z)\Sigma \D_{ij}^{-1}(z)\br_i\big|^q\nonumber\\
&\leq&\frac{K}{n^{q/2}},\label{gamma-varep-bound}
\end{eqnarray}
where the last inequality uses the boundedness of $\E|\beta_{ij}(z)|^q$ and $\E|\br_i'\D_{ij}^{-1}(z)\Sigma \D_{ij}^{-1}(z)\br_i|^q$.
From \eqref{ineq-ex} and \eqref{gamma-varep-bound}, we get
\begin{equation}\label{varep-bound}
\E|\varepsilon_j(z)|^q\leq Kn^{-q/2}\quad \text{and}\quad \E|\gamma_j(z)|^q\leq Kn^{-q/2},
\end{equation}
for $q$ even.

Next we show that $b_n(z)$ is bounded for all $n$. By the equality $b_n(z)-\beta_j(z)=b_n(z)\beta_j(z)\gamma_j(z)$ and  the boundedness of $\E|\beta_j(z)|^q$ and $\E|\gamma_j|^q$, we have
$$
|b_n(z)|=|\E\beta_j(z)+\E\beta_j(z)b_j(z)\gamma_j(z)|\leq K_1+K_2|b_n(z)|n^{-1/2}
$$
and thus, for all $n$ large enough,
\begin{equation}\label{bn-bound}
|b_j(z)|\leq \frac{K_1}{1-K_2n^{-1/2}}<K.
\end{equation}

Now we prove \eqref{tightness}. From the martingale decomposition and \eqref{D-D}, we have
\begin{eqnarray*}
	\frac{M_n^{(1)}(z_1)-M_n^{(1)}(z_2)}{z_1-z_2}&=&\sum_{j=1}^n(\E_j-\E_{j-1})\tr \D^{-1}(z_1)\D^{-1}(z_2)\\
	&=&\sum_{j=1}^n(\E_j-\E_{j-1})\left(\tr \D^{-1}(z_1)\D^{-1}(z_2)-\tr \D_j^{-1}(z_1)\D_j^{-1}(z_2)\right)\\
	&=&\sum_{j=1}^n(\E_j-\E_{j-1})\beta_j(z_1)\beta_j(z_2)\left(\br_j'\D_j^{-1}(z_1)\D_j^{-1}(z_2)\br_j\right)^2\\
	&&-\sum_{j=1}^n(\E_j-\E_{j-1})\beta_j(z_1)\br_j'\D_j^{-2}(z_1)\D_j^{-1}(z_2)\br_j\\
	&&-\sum_{j=1}^n(\E_j-\E_{j-1})\beta_j(z_2)\br_j'\D_j^{-1}(z_1)\D_j^{-2}(z_2)\br_j \\
	&:=&A_1+A_2+A_3.
\end{eqnarray*}
It is then enough to show $\E|A_1|^2$, $\E|A_2|^2$, and $\E|A_3|^2$ are all bounded.
The arguments for the boundedness are all similar to those in pages 582 and 583 in \cite{BS04},
and hence we only present the details for $\E|A_1|^2$ for illustration.

Replacing $\beta_j(z)$ in $R_1$ with $\beta_j(z)=b_n(z)-b_n(z)\beta_j(z)\gamma_j(z),$
we may obtain $A_1=A_{11}-A_{12}-A_{13}$ where
\begin{eqnarray*}
	A_{11}&=&\sum_{j=1}^nb_n(z_1)b_n(z_2)(\E_j-\E_{j-1})\left(\br_j'\D_j^{-1}(z_1)\D_j^{-1}(z_2)\br_j\right)^2,\\
	A_{12}&=&\sum_{j=1}^nb_n(z_1)b_n(z_2)(\E_j-\E_{j-1})\beta_j(z_1)\gamma_j(z_1)\left(\br_j'\D_j^{-1}(z_1)\D_j^{-1}(z_2)\br_j\right)^2,\\
	A_{13}&=&\sum_{j=1}^nb_n(z_2)(\E_j-\E_{j-1})\beta_j(z_1)\beta_j(z_2)\gamma_j(z_2)\left(\br_j'\D_j^{-1}(z_1)\D_j^{-1}(z_2)\br_j\right)^2.
\end{eqnarray*}
From \eqref{ineq-ex}, \eqref{ab-bound}, and \eqref{bn-bound},
\begin{eqnarray*}
	\E|A_{11}|^2&=&\E\bigg|\sum_{j=1}^nb_n(z_1)b_n(z_2)(\E_j-\E_{j-1})\big[\big(\br_j'\D_j^{-1}(z_1)\D_j^{-1}(z_2)\br_j\big)^2\\
	&&-\frac{1}{n^2}\big(\tr \Sigma \D_j^{-1}(z_1)\D_j^{-1}(z_2)\big)^2\big]\bigg|^2,\\
	&\leq&K\sum_{j=1}^n\E\bigg|\br_j'\D_j^{-1}(z_1)\D_j^{-1}(z_2)\br_j-\frac{1}{n}\tr \Sigma \D_j^{-1}(z_1)\D_j^{-1}(z_2)\bigg|^2\\
	&\leq&K.
\end{eqnarray*}
Using \eqref{ineq-ex}, \eqref{ab-bound},\eqref{varep-bound}, and \eqref{bn-bound},
\begin{eqnarray*}
	\E|A_{12}|^2&=&\sum_{j=1}^nb_n^2(z_1)b_n^2(z_2)\E\bigg|(\E_j-\E_{j-1})\beta_j(z_1)\gamma_j(z_1)\left(\br_j'\D_j^{-1}(z_1)\D_j^{-1}(z_2)\br_j\right)^2\big]\bigg|^2,\\
	&\leq&K\sum_{j=1}^n\left[\E\big|\gamma_j(z_1)\big|^2+v^{-10}p^2P\left(||\B_n||\geq\eta_r \ \text{or}\ \lambda_{\min}^{\tilde \B}\leq\eta_l\right)\right]\\
	&\leq&K.
\end{eqnarray*}
Similarly, we may get $\E|A_{13}|^2<K$. Hence the tightness of $M_n^{(1)}(z)$ is obtained.

\textit{Step 3: Convergence of $M_n^{(2)}(z)$.}
To finish the proof, it is enough to show that the sequence of $M_n^{(2)}(z)$ is bounded and equicontinuous, and converges to the mean function of the lemma for $z \in\mathcal C_n$.
The boundedness and equicontinuity can be verified following the arguments on pages 592 and 593 of \cite{BS04}, and thus we only focus  on the convergence of $M_n^{(2)}(z)$.

We first list some results that will be used in the sequel:
\begin{align}
&\sup_{n,z\in\mathcal C_n}|b_n(z)+z\um_0(z)|\to0,\quad\sup_{n,z\in\mathcal C_n}||zI- b_{n}(z)\Sigma||<\infty,\label{supsup}\\
&\sup_{n,z\in\mathcal C_n}\E|\tr \D^{-1}(z)\bM-\E\tr \D^{-1}(z)\bM|^2\leq K||\bM||^2,\label{tr-dm-bound}
\end{align}
where $\bM$ is any nonrandom $p\times p$ matrix.
These results can be verified step by step following similar discussions in \cite{BS04} and we omit the details.

Writing $\bV(z)=zI- b_{n}(z)\Sigma$, we decompose $M_n^{(2)}(z)$ as
\begin{align}
M_n^{(2)}(z)&=[p\E m_n (z)+\tr \bV^{-1}(z)]-[\tr \bV^{-1}(z)+pm_0(z)]:= S_{n}(z)-T_n(z)\label{pq}\\
&=[n\E\um_n (z)+nb_n(z)/z]-[nb_n(z)/z+n\um_0(z)]:= \uS_n(z)-\uT_n(z)\label{upq}.
\end{align}
Notice that
\begin{eqnarray*}
	T_n(z)&=&p\int \frac{d H_p(t)}{z- b_{n}(z)t}-p\int \frac{d H_p(t)}{z+z\um_{0}(z)t}\\
	&=&p\left[ b_{n}(z)+z\um_{0}(z)\right]\int \frac{td H_p(t)}{(z- b_{n}(z)t)(z+z\um_{0}(z)t)}\\
	&=&c_n\uT_n(z)\int \frac{td H_p(t)}{(z- b_{n}(z)t)(1+\um_{0}(z)t)}.
\end{eqnarray*}
We have
\begin{eqnarray*}
	\frac{M_n^{(2)}(z)-S_n(z)}{M_n^{(2)}(z)-\uS_n(z)}&=&\frac{T_n(z)}{\uT_n(z)}=\frac{c_n}{z}\int \frac{td H_p(t)}{(1+\um_{0}(z)t)^2}+o(1),
\end{eqnarray*}
where the second equality uses the convergence in \eqref{supsup}.

Our next task is to study the limits of $S_n(z)$ and $\uS_n(z)$. For simplicity, we suppress the expression $z$ when it is served as independent variables of some functions in the sequel.
All expressions and convergence statements hold uniformly for $z\in \mathcal C_n$.

We first simplify the expression of $S_n$. Using the identity $\br_j'\D^{-1}=\br_j'\D_j^{-1}\beta_j$, we have
\begin{eqnarray}
S_n&=&\E\tr(\D^{-1}+\bV^{-1})\nonumber\\
&=&\E\tr\left[\bV^{-1}\left(\sum_{j=1}^n\br_j\br_j'- b_{n}\Sigma\right)\D^{-1}\right]\nonumber\\
&=&n\E \beta_1\br_1'\D_1^{-1}\bV^{-1}\br_1-b_{n}\E\tr\Sigma \D^{-1}\bV^{-1}.\label{pn}
\end{eqnarray}
From \eqref{D-D} and $\beta_1=b_n-b_n\beta_1\gamma_1$,
\begin{eqnarray*}
	\E\tr \bV^{-1}\Sigma (\D_1^{-1}-\D^{-1})
	&=&\E\tr \bV^{-1}\Sigma \D_1^{-1}\br_1\br_1'\D_1^{-1}\beta_1\nonumber\\
	&=&b_n\E(1-\beta_1\gamma_1) \br_1'\D_1^{-1}\bV^{-1}\Sigma \D_1^{-1}\br_1,
\end{eqnarray*}
where $|\E\beta_1\gamma_1\br_1'\D_1^{-1}\bV^{-1}\Sigma \D_1^{-1}\br_1|\leq Kn^{-1/2}$.
From this and \eqref{pn}, we get
\begin{eqnarray*}
	S_n=n\E \beta_1\br_1'\D_1^{-1}\bV^{-1}\br_1-b_{n}\E\tr\Sigma \D_1^{-1}\bV^{-1}+\frac{1}{n}b_n^2\E\tr \D_1^{-1}\bV^{-1}\Sigma \D_1^{-1}\Sigma+o(1).
\end{eqnarray*}
Plugging $\beta_1=b_n-b_n^2\gamma_1+b_n^3\gamma_1^2-\beta_1b_n^3\gamma_1^3$ into the first term in the above equation, we obtain
\begin{eqnarray*}
	n\E \beta_1\br_1'\D_1^{-1}\bV^{-1}\br_1&=&b_n\E \tr \D_1^{-1}\bV^{-1}\Sigma-nb_n^2\E\gamma_1\br_1'\D_1^{-1}\bV^{-1}\br_1\\
	&&+nb_n^3\E\gamma_1^2\br_1'\D_1^{-1}\bV^{-1}\br_1-nb_n^3\E\gamma_1^3\br_1'\D_1^{-1}\bV^{-1}\br_1.
\end{eqnarray*}
Note that, from \eqref{ineq-ex}, \eqref{varep-bound}, and \eqref{tr-dm-bound},
\begin{eqnarray*}
	\E\gamma_1\br_1'\D_1^{-1}\bV^{-1}\br_1
	&=&\E\left[\br_1'\D_1^{-1}\br_1-\frac{1}{n}\tr \D_1^{-1}\Sigma\right]\left[\br_1'\D_1^{-1}\bV^{-1}\br_1-\frac{1}{n}\tr \D_1^{-1}\bV^{-1}\Sigma\right]\\
	&&+\frac{1}{n}\Cov(\tr \D_1^{-1}\Sigma, \tr \D_1^{-1}\bV^{-1}\Sigma)\\
	&=&\E\left[\br_1'\D_1^{-1}\br_1-\frac{1}{n}\tr \D_1^{-1}\Sigma\right]\left[\br_1'\D_1^{-1}\bV^{-1}\br_1-\frac{1}{n}\tr \D_1^{-1}\bV^{-1}\Sigma\right]\\
	&&+o\left(\frac{1}{n}\right),\\
	\E\gamma_1^2\br_1'\D_1^{-1}\bV^{-1}\br_1
	&=&\E\gamma_1^2\left[\br_1'\D_1^{-1}\bV^{-1}\br_1-\frac{1}{n}\tr \D_1^{-1}\bV^{-1}\Sigma\right]\\
	&&+\frac{1}{n}\Cov(\gamma_1^2, \tr \D_1^{-1}\bV^{-1}\Sigma)+\frac{1}{n}\E\gamma_1^2E\tr \D_1^{-1}\bV^{-1}\Sigma\\
	&=&\frac{1}{n}\E\gamma_1^2\E\tr \D_1^{-1}\bV^{-1}\Sigma+o\left(\frac{1}{n}\right),\\
	\E\gamma_1^3\br_1'\D_1^{-1}\bV^{-1}\br_1&=&o\left(\frac{1}{n}\right).
\end{eqnarray*}
We thus arrive at
\begin{eqnarray*}
	S_n
	&=&-nb_n^2\E\left[\br_1'\D_1^{-1}\br_1-\frac{1}{n}\tr \D_1^{-1}\Sigma\right]\left[\br_1'\D_1^{-1}\bV^{-1}\br_1-\frac{1}{n}\tr \D_1^{-1}\bV^{-1}\Sigma\right]\nonumber\\
	&&+b_n^3\E\gamma_1^2\E\tr \D_1^{-1}\bV^{-1}\Sigma +\frac{1}{n}b_n^2\E\tr \D_1^{-1}\bV^{-1}\Sigma \D_1^{-1}\Sigma +o(1).\label{spn}
\end{eqnarray*}

On the other hand, by the identity $\br_j'\D^{-1}=\br_j'\D_j^{-1}\beta_j$, we have
$$
p+z\tr \D^{-1}=\tr(\B_n\D^{-1})=\sum_{j=1}^n\beta_j\br_j'\D_j^{-1}\br_j=n-\sum_{j=1}^n\beta_j,
$$
which implies  $nz\um_n=-\sum_{j=1}^n\beta_j$. From this, together with $\beta_1=b_n-b_n^2\gamma_1+b_n^3\gamma_1^2-\beta_1b_n^3\gamma_1^3$, \eqref{ineq-ex}, we get
\begin{eqnarray*}\label{supn}
	\uS_n=-\frac{n}{z}\E\left(\beta_1- b_n\right)
	=-\frac{n}{z}b_n^3\E\gamma_1^2+o(1).
\end{eqnarray*}

Applying Lemma 2 to the simplified  $S_n$ and $\uS_n$, and then replacing $\D_j$ with $\D$ in the derived results
yield
\begin{eqnarray}
S_n
&=&-\frac{b_n^2}{n}\bigg[\E\tr \D^{-1}\Sigma \D^{-1}\bV^{-1}\Sigma+\frac{2}{p}\bigg( \gamma_2\E\tr\Sigma \D^{-1}\tr\Sigma\D^{-1}\bV^{-1}\nonumber\\
&&-\E\tr\Sigma^2\D^{-1}\tr\Sigma\D^{-1}\bV^{-1}-\E\tr\Sigma \D^{-1}\tr \Sigma^2\D^{-1}\bV^{-1}\bigg)\bigg]\nonumber\\
&&+\frac{2b_n^3}{n^2}\bigg[\E\tr \D^{-1}\Sigma \D^{-1}\Sigma+\frac{1}{p}\bigg( \gamma_2\E\tr\Sigma \D^{-1}\tr\Sigma\D^{-1}\nonumber\\
&&-2\E\tr\Sigma^2\D^{-1}\tr\Sigma\D^{-1}\bigg) \bigg] \E\tr \D^{-1}\bV^{-1}\Sigma+o(1),\label{fpn}\\
\uS_n&=&\frac{-2b_n^3}{zn}\bigg[\E\tr \D^{-1}\Sigma \D^{-1}\Sigma+\frac{1}{p}\bigg( \gamma_2\E\tr\Sigma \D^{-1}\tr\Sigma\D^{-1}\nonumber\\
&&-2\E\tr\Sigma^2\D^{-1}\tr\Sigma\D^{-1}\bigg) \bigg] +o(1).\label{fupn}
\end{eqnarray}

To study the limits of $S_n$ and $\uS_n$, we compare the difference between $\D^{-1}$ and $\bV^{-1}$. Similar to \eqref{dj-1}-\eqref{CM}, we have
\begin{eqnarray}
\D^{-1}+\bV ^{-1}
=b_n\R_1+\R_2+\R_3,\label{tilde-D-L}
\end{eqnarray}
where
$\R_1=\sum_{j=1}^n\bV ^{-1}(\br_j\br_j'-n^{-1}\Sigma)\D_{j}^{-1}$ and, for any $p\times p$ matrix $\bM$,
\begin{eqnarray}
|\E\tr \R_2\bM|\leq n^{1/2} K (\E||\bM||^4)^{1/4},\quad
|\tr \R_3\bM|\leq K(\E||\bM||^2)^{1/2}\label{tilde-CM}.
\end{eqnarray}
Moreover, for nonrandom $\bM$ with bounded norm,
\begin{eqnarray}
|\E\tr \R_1\bM|\leq n^{1/2}K\label{tilde-AM}.
\end{eqnarray}
Similar to \eqref{R123}, we write
\begin{equation}
\tr \R_1\Sigma \D^{-1}\bM=\tilde R_{11}+\tilde R_{12}+\tilde R_{13},
\end{equation}
where $\tilde R_{11}=\tr \sum_{j=1}^n\bV ^{-1}\br_j\br_j'\D_j^{-1}\Sigma(\D^{-1}-\D_j^{-1})\bM$,
$\E\tilde R_{12}=0$, and $|\E\tilde R_{13}|\leq K.$
Using \eqref{D-bound}, \eqref{ineq-ex}, and \eqref{tr-dm-bound}, we get
\begin{eqnarray}
\E\tilde R_{11}&=&-n\E\beta_1\br_1\D_1^{-1}\Sigma \D_1^{-1}\br_1\br_1'\D_1^{-1}\bM\bV ^{-1}\br_1\nonumber\\
&=&-b_nn^{-1}\E(\tr \D_1^{-1}\Sigma \D_1^{-1}\Sigma)( \tr \D_1^{-1}\bM\bV ^{-1}\Sigma)+o(1)\nonumber\\
&=&-b_nn^{-1}\E(\tr \D^{-1}\Sigma \D^{-1}\Sigma) (\tr \D^{-1}\bM\bV^{-1}\Sigma)+o(1)\nonumber\\
&=&-b_nn^{-1}\E(\tr \D^{-1}\Sigma \D^{-1}\Sigma)\E (\tr \D^{-1}\bM\bV ^{-1}\Sigma)+o(1).\label{er11}
\end{eqnarray}
From \eqref{tilde-D-L}-\eqref{er11} we get
\begin{eqnarray}
&&\frac{1}{n}\E\tr \D^{-1}\Sigma \D^{-1}\Sigma\nonumber\\
&=&-\frac{1}{n}\E\tr \bV^{-1}\Sigma \D^{-1}\Sigma-\frac{b_n^2}{n^2} \E \tr \D^{-1}\Sigma \D^{-1}\Sigma \E\tr \bV^{-1}\Sigma \D^{-1}\Sigma+o(1)\nonumber\\
&=&-\frac{1}{n}\E\tr \bV^{-1}\Sigma \D^{-1}\Sigma\left[1+\frac{b_n^2}{n} \E\tr \bV^{-1}\Sigma \D^{-1}\Sigma\right]^{-1}+o(1),\label{dsds}\\
&&\frac{1}{n}\E\tr \D^{-1}\Sigma \D^{-1}\bV^{-1}\Sigma\nonumber\\
&=&-\frac{1}{n}\E\tr \bV^{-1}\Sigma \D^{-1}\bV^{-1}\Sigma\left[1+\frac{b_n^2}{n}\E \tr \D^{-1}\Sigma \D^{-1}\Sigma\right]+o(1)\nonumber\\
&=&-\frac{1}{n}\E\tr \bV^{-1}\Sigma \D^{-1}\bV^{-1}\Sigma\left[1+\frac{b_n^2}{n} \E\tr \bV^{-1}\Sigma \D^{-1}\Sigma\right]^{-1}+o(1).\label{dsdm}
\end{eqnarray}
From \eqref{bnz}, \eqref{tilde-D-L}-\eqref{dsdm} we get
\begin{eqnarray*}
	\frac{1}{n}\E\tr \D^{-1}\Sigma^k&=&-\int\frac{c_nt^kd H_p(t)}{z(1+\um_{0}t)}+o(1),\quad k=1,2,\\
	\frac{1}{n}\E\tr \D^{-1}\bV^{-1}\Sigma^k &=&-\int\frac{c_nt^kd H_p(t)}{z^2(1+\um_{0}t)^2}+o(1),\quad k=1,2,\\
	\frac{1}{n}\E\tr \D^{-1}\Sigma \D^{-1}\Sigma&=&\int\frac{c_nt^2d H_p(t)}{z^2(1+\um_{0}t)^2}\left[1-\int\frac{c_n\um_0^2t^2d H_p(t)}{(1+\um_0t)^2}\right]^{-1}+o(1),\\
	\frac{1}{n}\E\tr \D^{-1}\Sigma \D^{-1}\bV^{-1}\Sigma&=&\int\frac{c_nt^2d H_p(t)}{z^3(1+\um_0t)^3}\bigg[1-\int\frac{c_n\um_0^2t^2d H_p(t)}{(1+\um_0t)^2}\bigg]^{-1}+o(1).
\end{eqnarray*}

Combining the above results with \eqref{fpn} and \eqref{fupn}, we obtain
\begin{eqnarray*}
	S_n
	&=&-\int\frac{c_n\um_0^2t^2d H_p(t)}{z(1+\um_0t)^3}\bigg[1-\int\frac{c_n\um_0^2t^2d H_p(t)}{(1+\um_0t)^2}\bigg]^{-1}\\
	&&-2\bigg[\int\frac{ \gamma_2t-t^2d H_p(t)}{z(1+\um_{0}t)}\int\frac{c_n\um_0^2td H_p(t)}{(1+\um_{0}t)^2}
	-\int\frac{td H_p(t)}{z(1+\um_{0}t)}\int\frac{c_n\um_0^2t^2d H_p(t)}{(1+\um_{0}t)^2}\bigg]\\
	&&+\frac{2c_n^2\um_0^3}{z}\bigg\{\int\frac{t^2d H_p(t)}{(1+\um_{0}t)^2}\left[1-\int\frac{c_n\um_0^2t^2d H_p(t)}{(1+\um_0t)^2}\right]^{-1}\\
	&&+ \gamma_2\bigg[\int\frac{td H_p(t)}{1+\um_{0}t}\bigg]^2-2\int\frac{t^2d H_p(t)}{1+\um_{0}t}\int\frac{td H_p(t)}{1+\um_{0}t} \bigg\} \int\frac{td H_p(t)}{(1+\um_{0}t)^2}+o(1),\\
	\uS_n&=&2c_n\um_0^3\bigg\{\int\frac{t^2d H_p(t)}{(1+\um_{0}t)^2}\left[1-\int\frac{c_n\um_0^2t^2d H_p(t)}{(1+\um_0t)^2}\right]^{-1}+ \gamma_2\bigg[\int\frac{td H_p(t)}{1+\um_{0}t}\bigg]^2\nonumber\\
	&&-2\int\frac{t^2d H_p(t)}{1+\um_{0}t}\int\frac{td H_p(t)}{1+\um_{0}t} \bigg\} +o(1).
\end{eqnarray*}
Therefore we get
\begin{eqnarray*}
	M_n^{(2)}(z)&=&\frac{S_n-\uS_nT_n/\uT_n}{1-T_n/\uT_n}\\
	&\to&\bigg[1-\int \frac{ctd H(t)}{z(1+\um t)^2}\bigg]^{-1}\bigg\{\int\frac{c\um^2t^2d H(t)}{z(1+\um t)^3}\bigg[1-\int\frac{c\um^2 t^2d H(t)}{(1+\um t)^2}\bigg]^{-1}\\
	&&-\frac{2c\um^2}{z}\bigg[\int\frac{ \gamma_2t-t^2d H(t)}{1+\um t}\int\frac{td H(t)}{(1+\um t)^2}-\int\frac{td H(t)}{1+\um t}\int\frac{t^2d H(t)}{(1+\um t)^2}\bigg]\bigg\},
\end{eqnarray*}
as $n\to\infty$. Using the identity
\begin{equation}\label{uid}
\bigg[1-\int \frac{ctd H(t)}{z(1+\um t)^2}\bigg]^{-1}=-z\um\bigg[1-\int\frac{c\um^2 t^2d H(t)}{(1+\um t)^2}\bigg]^{-1}=-\frac{z\um'}{\um}
\end{equation}
we finally obtain the mean function of the lemma.

\end{proof}

\subsection{Proof of Theorem \ref{lsd}}

Following Theorem 1.1 in
\cite{BZ08}, it is sufficient to show that,
for any bounded sequence of symmetric matrices
$\{\C_p\}$,
\begin{equation}\label{BZc}
Var(\y' \C_p\y)=o(p^2).
\end{equation}
Write $\y=\sqrt{p}\A\bu/||\A\bu||=\sqrt{p}\A\z/||\A\z||$ where $\z\sim N(0, I_p)$. Since the eigenvalues of the SSCM $\B_n$ are invariant under orthogonal transformation, it's enough to consider the diagonal matrix $\A$.
Therefore, by taking $\C=\tilde \C=\C_p$ in Lemma \ref{double-e}, one can verify the condition \eqref{BZc}.

\subsection{Proof of Theorem \ref{clt}}

For any distribution function $G$ and function $f$ analytic on a simple connected domain $D$ containing the support of $G$, it holds that
\begin{equation}
\int f(x)dG(x)=-\frac1{2\pi {\rm i}}\oint\limits_{\mathcal C} f(z)m_G(z)dz\label{cf}
\end{equation}
where $m_G(z)$ denotes the Stieltjes transform of $G$ and $\mathcal C\subset D$ is a simple, closed, and  positively oriented contour enclosing
the support of $G$.  Similar to \eqref{cont}, we choose $v_0$, $x_r$, and
$x_l$ such that $f_1,\ldots,f_k$ are all analytic on and inside the
contour $\mathcal C$. We denote by $K$ a common upper bound of these functions on $\mathcal C$.
Therefore, almost surely, for all $n$ large, $\{f_1,\ldots, f_k\}$ satisfy the equation in \eqref{cf} with $G=F^{B_n}$ and moreover,
\begin{eqnarray*}
	\left|\int f_i(z)(M_n(z)-\widehat M_n(z))dz\right|
	&\leq& 4K\varepsilon_n\big(| \max\{\lambda_{\max}^{\Sigma}(1+\sqrt
	c_n)^2,\lambda_{\max}^{\B_n}\}-x_r|^{-1} \hfill\\ \hfill
	&&+|\min\{\lambda_{\min}^{\Sigma}I_{(0,1)}(c_n)(1-\sqrt c_n)^2,
	\lambda_{\min}^{\B_n}\}-x_l|^{-1}\big)\nonumber
\end{eqnarray*}
which converges to zero
as $n\to\infty$. Since
$$\widehat M_n(\cdot)\longrightarrow
\left(-\frac1{2\pi {\rm i}}\int f_1(z)\,\widehat M_n(z)dz,\ldots,
-\frac1{2\pi {\rm i}}\int f_k(z)\,\widehat M_n(z)dz \right)$$
is a continuous mapping of $C(\mathcal C ,{\mathbb R}^2)$ into ${\mathbb R}^{k}$, it follows from Lemma \ref{clt-mn} that
the above random vector converges to a multivariate Gaussian vector $(X_{f_1},\ldots,X_{f_k})$ whose mean and covariance functions are
\begin{eqnarray*}
	\E(X_f)&=&-\frac{1}{2\pi\rm i}\oint_{\mathcal C_1} f(z)\E[M(z)]dz,\\
	Cov\left(X_f,X_g\right)
	&=&-\frac{1}{4\pi^2}\oint_{\mathcal C_1}\oint_{\mathcal C_2} f(z_1)g(z_2)Cov[M(z_1), M(z_2)]dz_1dz_2,
\end{eqnarray*}
where $f, g \in \{f_1,\ldots, f_k\}$ and $\{\mathcal C_1, \mathcal C_2\}$ are two non-overlapping analogues  of the contour $\mathcal C$.

From the following two identities
\begin{eqnarray*}
	&&\frac{1}{2\pi{\rm i}}\oint_{\mathcal C_1} f(z)\left(\um(z)+z\um'(z)\right)dz=-\frac{1}{2\pi{\rm i}}\oint_{\mathcal C_1} zf'(z)\um(z)dz= c\int xf'(x)dF(x),\\
	&&\frac{1}{2\pi{\rm i}}\oint_{\mathcal C_1} f(z)\left(\frac{\um'(z)}{\um^2(z)}-1\right)dz=\frac{1}{2\pi{\rm i}}\oint_{\mathcal C_1} \frac{f(z)\um'(z)}{\um^2(z)}dz,
\end{eqnarray*}
we obtain the form of the limiting covariance function in the theorem.

\subsection{Proof of Corollary \ref{clt-mom}}
Choose a contour $\mathcal C$ for the integrals such that
$\max_{t\in S_{ H}, z\in \mathcal C}|t\um(z)|<1,$
where $S_{ H}$ is the support of $H$.
Let
$\um(\mathcal C)=\{\um(z):z\in \mathcal C\}$
denote the image of $\mathcal C$ under $\um(z)$. Then $\um(\mathcal C)$ is a simple and closed contour having clockwise direction and enclosing zero \citep{Q17}.

By the identity in \eqref{mp},
the integral in the mean function of Theorem \ref{clt} becomes
\begin{eqnarray*}
	v_j
	&=&-\frac{c}{2\pi{\rm i}}\oint_{\um(\mathcal C)}\frac{P^j(\um)P_{2,3}(\um)}{\um^{j-1}(1-c\um^2P_{2,2}(\um))}d\um-\frac{c \gamma_2}{\pi{\rm i}}\oint_{\um(\mathcal C)}\frac{P^j(\um)P_{1,1}(\um)P_{1,2}(\um)}{\um^{j-1}}d\um\nonumber\\
	&&+\frac{c}{\pi{\rm i}}\oint_{\um(\mathcal C)}\frac{P^j(\um)P_{2,1}(\um)P_{1,2}(\um)}{\um^{j-1}}d\um+\frac{c}{\pi{\rm i}}\oint_{\um(\mathcal C)}\frac{P^j(\um)P_{1,1}(\um)P_{2,2}(\um)}{\um^{j-1}}d\um.
\end{eqnarray*}
From this and the Cauchy integral theorem, we get the mean function. The covariance function can be obtained following the proof of Theorem 1 in \cite{Q17}.

\section{Appendix}\label{app}

\begin{lemma}\label{ineq}
	For any $p\times p$ complex matrix $\C$ and $\y=\sqrt{p}\bx/||\bx||$ with $\bx\sim N(0, \Sigma)$ and $||\Sigma||\leq 1$,
	\begin{equation}\label{q-moment}
	\E\left|\y'\C\y-\tr \Sigma \C\right|^q\leq K_q||\C||^qp^{q/2},\quad q\geq 2,
	\end{equation}
	where $K_q$ is a positive constant depending only on $q$.
\end{lemma}
\begin{proof}
	This lemma follows from Lemma 2.2 in \cite{BS04} and similar arguments in the proof of Lemma 5 in \cite{Gao16}.
\end{proof}

\begin{lemma}[\cite{B73}]\label{mi1}
	Let $\{X_k\}$ be a complex martingale difference sequence with respect to the increasing $\sigma$-field $\{\mathcal F_k \}$. Then, for $q \geq 2$,
	$$
	\E\bigg|\sum X_k\bigg|^q\leq K_q\left\{\E\left(\sum \E\left(|X_k|^2|\mathcal F_{k-1}\right)\right)^{q/2}+E\left(\sum |X_k|^q\right)\right\}.
	$$
\end{lemma}

\begin{lemma}[\cite{B73}]\label{mi2}
	Let $\{X_k\}$ be a complex martingale difference sequence with respect to the increasing $\sigma$-field $\{\mathcal F_k \}$. Then, for $q > 1$,
	$$
	\E\bigg|\sum X_k\bigg|^q\leq K_q\E\left(\sum |X_k|^2\right)^{q/2}.
	$$
\end{lemma}

\begin{lemma}[Theorem 35.12 of \cite{B95}]\label{MCLT}
	Suppose for each $n$ $Y_{n1}, Y_{n2},\ldots Y_{n\br_n}$ is a real martingale difference sequence with respect to the increasing $\sigma$-field ${\{\mathcal F_{nj}\}}$ having second moments.
	If  for each $\varepsilon>0$,
	\begin{eqnarray*}
		\sum_{j=1}^{\br_n}\E(Y_{nj}^2I_{(|Y_{nj}|\geq\varepsilon)})\rightarrow0\quad\text{and}\quad \sum_{j=1}^{\br_n}\E(Y_{nj}^2|\mathcal F_{n,j-1})\xrightarrow{i.p.}\sigma^2,
	\end{eqnarray*}
	as $n\rightarrow\infty$, where $\sigma^2$ is a positive constant, then
	$$
	\sum_{j=1}^{\br_n}Y_{nj}\xrightarrow{D}N(0,\sigma^2).
	$$
\end{lemma}

\begin{lemma}\label{lambda-bound}
	Suppose that Assumptions (a)-(c) hold. Then, for any $s$ positive,
	$$
	P(||\B_n||>\eta_r)=o(n^{-s}),
	$$
	whenever $\eta_r>\lim\sup_{p\rightarrow\infty} ||\Sigma||(1+\sqrt{c})^2$.
	If $0<\lim\inf_{p\rightarrow\infty}\lambda_{\min}^{\Sigma}I_{(0,1]}(c)$ then,
	$$
	P(\lambda_{\min}^{\B_n}<\eta_l)=o(n^{-s}),
	$$
	whenever $0<\eta_l<\lim\inf_{p\rightarrow\infty}\lambda_{\min}^{\Sigma}I_{(0,1)}(c)(1-\sqrt{c})^2$.
\end{lemma}

\begin{proof}
	Let $\bx_j=\A\z_j$ where $\A\A'=\T$ and $\z_j\sim N(0, I_p)$, $j=1,\ldots,n.$ Also let $\B_n^{(0)}=(1/n)\sum_{j=1}^n\A\z_j\z_j'\A'$.
	From \cite{BS04}, the conclusions of this lemma hold when $(\B_n, \Sigma)$ are replaced with $(\B_n^{(0)}, \T)$.
	Choose $\eta_r^{(0)}$ and $\eta_l^{(0)}$ satisfying
	$$
	\eta_l<r_1^{-1}\eta_l^{(0)}<\lim\inf_{p\rightarrow\infty}\lambda_{\min}^{\Sigma}I_{(0,1)}(c)(1-\sqrt{c})^2\quad \text{and}\quad \lim\sup_{p\rightarrow\infty} ||\Sigma||(1+\sqrt{c})^2< r_1^{-1}\eta_r^{(0)}<\eta_r,
	$$
	where $r_1=\tr(\T)/p$. From Lemma 1, we have
	$$
	\eta_l^{(0)}<\lim\inf_{p\rightarrow\infty}\lambda_{\min}^{\T}I_{(0,1)}(c)(1-\sqrt{c})^2 \quad \text{and}\quad   \lim\sup_{p\rightarrow\infty} ||\T||(1+\sqrt{c})^2<\eta_r^{(0)}.
	$$
	
	Using inequalities
	$$
	\min_{1\leq j\leq n}\frac{p}{||\A\z_j||^2}\lambda_{\min}^{\B_n^{(0)}}\leq \lambda_{\min}^{\B_n}\leq ||\B_n||\leq\max_{1\leq j\leq n}\frac{p}{||\A\z_j||^2}||\B_n^{(0)}||,
	$$
	we may get
	\begin{eqnarray*}
		P(||\B_n||>\eta_r)&\leq& P\left(||\B_n^{(0)}||>\eta_r^{(0)}\right)+P\left(\max_{1\leq j\leq n}\frac{p}{||\A\z_j||^2}||\B_n^{(0)}||>\eta_r, ||\B_n^{(0)}||\leq \eta_r^{(0)}\right)\\
		&\leq& P\left(\max_{1\leq j\leq n}\frac{p}{||\A\z_j||^2}>\frac{\eta_r}{\eta_r^{(0)}}\right)+o(n^{-s})\\
		&\leq& nP\left(\bigg|\frac{||\A\z_1||^2}{p}-r_1\bigg|>r_1-\frac{\eta_r^{(0)}}{\eta_r}\right)+o(n^{-s}),\\
		&=&o(n^{-s}),
	\end{eqnarray*}
	where the last equality is from the Chebyshev inequality and the fact $r_1>\eta_r^{(0)}/\eta_r$.
	Similarly, $P(\lambda_{\min}^{\B_n}<\eta_l)=o(n^{-s}).$
	
\end{proof}

\end{document}